\documentclass[11pt]{amsart}
\usepackage{enumerate, mathrsfs, amsmath, amssymb}
\usepackage{paralist}
\usepackage{xy, diagxy}

\theoremstyle{plain}
\newtheorem{prop}{Proposition}[section]
\newtheorem{theorem}[prop]{Theorem}
\newtheorem{theoremintro}{Theorem}

\newtheorem{lemma}[prop]{Lemma}
\newtheorem{cor}[prop]{Corollary}
\newtheorem{corintro}[theoremintro]{Corollary}
\theoremstyle{definition}
\newtheorem*{defn*}{Definition}
\newtheorem{defn}[prop]{Definition}
\newtheorem{remark}[prop]{Remark}
\newtheorem*{remark*}{Remark}

\newtheorem{example}[prop]{Example}
\newtheorem*{example*}{Example}
\theoremstyle{remark}
\numberwithin{equation}{section}
\newcommand{\PSL}{\mathsf{PSL}}
\newcommand{\AGL}{\mathsf{AGL}}
\newcommand{\AGammaL}{\mathsf{A\Gamma L}}
\newcommand{\GF}{\mathsf{GF}}

\newcommand{\ZZ}{\mathbf{Z}}
\newcommand{\RR}{\mathbf{R}}
\newcommand{\norma}{\mathcal{N}}

\newcommand{\la}{\langle}
\newcommand{\ra}{\rangle}
\newcommand{\inv}{^{-1}}
\newcommand{\mS}{\mathcal{S}}
\newcommand{\acomm}{\mathscr{L}}
\newcommand{\bd}{\partial}

\DeclareMathOperator{\Aut}{Aut}
\DeclareMathOperator{\AAut}{AAut}
\DeclareMathOperator{\Inn}{Inn}
\DeclareMathOperator{\Out}{Out}
\DeclareMathOperator{\Gal}{Gal}

\DeclareMathOperator{\Alt}{Alt}
\DeclareMathOperator{\Sym}{Sym}
\DeclareMathOperator{\rist}{rst}
\DeclareMathOperator{\st}{st}
\DeclareMathOperator{\Comm}{Comm}
\DeclareMathOperator{\comm}{ad}
\DeclareMathOperator{\QZ}{QZ}

\begin{document}

\title[Simple locally compact groups acting on trees]{Simple locally
compact groups acting on trees and their germs of automorphisms}

\author[Pierre-Emmanuel Caprace]{Pierre-Emmanuel Caprace*}
\address{UCLouvain, Chemin du Cyclotron 2, 1348 Louvain-la-Neuve, Belgium}
\email{pe.caprace@uclouvain.be}
\thanks{*F.N.R.S. Research Associate}
\author[Tom De Medts]{Tom De Medts}
\address{Ghent University, Dept. of Pure Mathematics and Computer Algebra,
\mbox{Krijgslaan} 281 -- S22, 9000 Gent, Belgium}
\email{Tom.DeMedts@UGent.be}

\date{January 2011}
%\LaTeX ed on \today}%
\keywords{tree, totally disconnected locally compact group, simple group, primitive group, spheromorphisms, germ of automorphism, branch group, commensurator}
\subjclass[2000]{(MSC2010): 20E08, 22A05, 22D05}

\begin{abstract}
Automorphism groups of locally finite trees provide a large class of examples of simple totally disconnected
locally compact groups.
It is desirable to understand the connections between the global and local structure of such a group.
Topologically, the local structure is given by the commensurability class of a vertex stabiliser;
on the other hand, the action on the tree suggests that the local structure should correspond to the local action
of a stabiliser of a vertex on its neighbours.

We study the interplay between these different aspects for the special class of groups satisfying
Tits' independence property.
We show that such a group has few open subgroups if and only if it acts locally primitively.
Moreover, we show that it always admits many germs of automorphisms which do not extend to automorphisms,
from which we deduce a negative answer to a question by George Willis.
Finally, under suitable assumptions, we compute the full group of germs of automorphisms;
in some specific cases, these turn out to be simple and compactly generated,
thereby providing a new infinite family of examples which generalise Neretin's group of spheromorphisms. Our methods describe more generally the abstract commensurator group for a large family of self-replicating profinite branch groups.
\end{abstract}

\maketitle

%%%%%%%%%%%%%%%%%%%%%%%%%%%%%%%%%%%%%%%%%%%%%%%%%%%%%%%%%%%%%%%%%%%%%%%%%%%%%%
\section{Introduction}\label{sec:intro}
%%%%%%%%%%%%%%%%%%%%%%%%%%%%%%%%%%%%%%%%%%%%%%%%%%%%%%%%%%%%%%%%%%%%%%%%%%%%%%

Recent advances in the structure theory of locally compact groups
bring new stimulus to the investigation of the class of locally compact groups
which are non-discrete,
topologically simple and compactly generated. The hope for progress in
this direction is based on
the one hand on a fairly satisfactory understanding of the connected case,
which is completely described thanks to the solution to Hilbert fifth
problem, and on the other
hand to recent progress on the theory of totally disconnected
locally compact groups which tends to show that there is a sharp
contrast between the structure of discrete and non-discrete groups. This
has been notably illustrated by George Willis {\em et.\@ al.\@} in a series of works
starting with \cite{Willis94} and by Marc Burger and Shahar Mozes in their
thorough study of non-discrete automorphism groups of locally finite trees
\cite{Burger-Mozes1}. It was moreover shown in
\cite{Caprace-Monod-monolith} that, under some
natural restrictions which exclude the existence of infinite
discrete quotients, a compactly generated locally compact group
decomposes into finitely many pieces (namely subquotients) which are
all compact,  or compactly generated abelian, or  compactly
generated and topologically simple. This provides additional
motivation to focus on simple groups in this category.

Let us denote by $\mS$ the class of non-discrete totally disconnected
locally compact groups which are  topologically simple. The compactly
generated groups belonging to $\mS$ and known to us at the time of writing
fall into the following classes\footnote{Observe that the examples in this
list fall into countably many isomorphism classes; we do not know whether $\mS$
contains uncountably many pairwise non-isomorphic compactly generated
groups. In the discrete case, it is known since  \cite{Camm} that there
are uncountably many isomorphism classes of finitely generated  simple
groups.}
:

\begin{itemize}
\item semi-simple linear algebraic groups over local fields (including
groups of mixed type in the sense of Tits);

\item complete Kac--Moody groups over finite fields;

\item tree-automorphism groups  and their avatars.
\end{itemize}

The groups from this last class include the tree-automorphism groups
satisfying a simplicity criterion established by Tits~\cite{Tits:trees}
and recalled in Section~\ref{sec:Tits} below. This criterion shall be
referred to as  \textbf{Tits' independence property}.  The aforementioned
 avatars include automorphism groups of some CAT(0) cube complexes
\cite{Haglund-Paulin} and Neretin's group of tree-spheromorphisms
\cite{Neretin}, whose simplicity is proved in \cite{Kapoudjian}.

\medskip
In this paper we focus on compactly generated locally compact groups which
act on trees and satisfy Tits' independence property. Following the
spirit of the work by Burger--Mozes~\cite{Burger-Mozes1}, the central
theme of our work is to investigate to what extent the global structure of
these groups is determined by their local structure. A first sample of this
phenomenon is provided by the following (see also Theorem~\ref{thm:LocPrim} below),
where $E(v)$ denotes the set of edges emanating from the vertex $v$.

\begin{theoremintro}\label{thmi:LocPrim}
Let $T$ be a tree all of whose vertices have valence~$\geq 3$ and $G \in \mS$ be a compactly generated closed subgroup of $\Aut(T)$ which does not stabilise any proper non-empty subtree and which satisfies Tits' independence property. Then the following conditions are equivalent.
\begin{enumerate}[\rm (i)]
\item Every proper open subgroup of $G$ is compact.

\item For every vertex $v \in V(T)$, 	the induced action of $G_v$ on $E(v)$ is primitive; in particular $G$ is edge-transitive.
\end{enumerate}
\end{theoremintro}

It turns out that, in the above setting, the $G_v$-action on $E(v)$ cannot
be cyclic (see Lemma~\ref{lem:LocCyclic} below). Moreover, the hypotheses of thickness of $T$ and of minimality of the $G$-action can be relaxed, see Theorem~\ref{thm:LocPrim} below. 

\begin{remark}
It was shown by Burger--Mozes \cite[Theorem~4.2]{Burger-Mozes2} that a closed subgroup $H \leq \Aut(T)$ which is locally $\infty$-transitive (or equivalently, a closed subgroup acting doubly transitively on the boundary $\partial T$, see \cite[Lemma~3.1.1]{Burger-Mozes1}) enjoys the \textbf{Howe--Moore property}. It is well known that in a locally compact group satisfying the Howe--Moore property, any proper open subgroup is compact. This is consistent with Theorem~\ref{thmi:LocPrim} since a locally $\infty$-transitive action is clearly locally $2$-transitive, and hence locally primitive. We do not know of any \emph{local} characterization of the Howe--Moore property for closed subgroups of $\Aut(T)$ (acting cocompactly on $T$); in particular, we do not know whether the condition that the action be locally $\infty$-transitive is necessary. 
\end{remark}

Theorem~\ref{thmi:LocPrim} relates the global structure of $G$ to the
structure of its \emph{maximal} compact subgroups. In the spirit of a
general theory inspired by the classical case of Lie groups, it would be
even more desirable to relate the structure of $G$ to arbitrarily small
compact open subgroups. In order to address this issue, we consider the
group $\mathscr L(G)$ of \textbf{germs of automorphisms} of $G$. By
definition, this group consists of isomorphisms between compact open
subgroups of $G$, modulo the equivalence relation which identifies
isomorphisms between pairs of compact open subgroups which coincide on
respective open subgroups. Alternatively the group $\mathscr L(G)$ can be
defined as  the group of \textbf{abstract commensurators} of any compact
open subgroup of $G$. The group $\mathscr L(G)$ is defined for any totally
disconnected locally compact group; it is trivial when $G$ is discrete.
Since every identity neighbourhood of $G$ contains some compact open
subgroup, the group $\mathscr L(G)$ is an invariant of $G$ determined by
its local structure.
The kernel of the canonical homomorphism $\comm : G \to \mathscr L(G)$ is precisely
the {\bf quasi-centre} $\QZ(G)$ of $G$. If the quasi-centre is closed, then
$\mathscr L(G)$ carries a natural group topology,
which is again totally disconnected and locally compact and such that the map $G \to \mathscr L(G)$ is continuous. If in addition $\QZ(G)$ is discrete, then the groups $G$ and $\mathscr L(G)$ are \textbf{locally isomorphic}, \emph{i.e.} they contain isomorphic compact open subgroups. In particular $\mathscr L(\mathscr L(G)) \cong \mathscr L(G)$. We refer to Section~\ref{sec:Germ} below for more details.

It is a natural question to ask to what extent the
group $G$ can be recovered from its local structure.
Following \cite{BarneaErshovWeigel}, we
say that a group $G \in \mS$ is \textbf{rigid} if
any isomorphism between two compact open subgroups of $G$ extends to a unique
automorphism of $G$.
An equivalent way to state this, is to say that
the canonical homomorphism $\Aut(G) \to \mathscr L(G)$ is an isomorphism.
If $G$ is compactly generated, then the group $G \cong \comm(G)$ can be identified with the intersection of
all non-trivial closed normal subgroups of $\acomm(G)$ (see Proposition~\ref{prop:AdmissibleRigid} below);
in this sense a rigid compactly generated group is thus determined by its local structure. 

On the other hand, it is clear that two groups containing isomorphic compact
open subgroups have isomorphic groups of germs of automorphisms. In order to address this issue, we shall
say that a group $G \in \mS$ is \textbf{Lie-reminiscent} if for any topologically simple group $H \in
\mS$ which is locally isomorphic to $G$, we have $H \cong G$. Using the ideas developed in 
\cite{BarneaErshovWeigel}, one can show that a compactly generated group $G \in \mS$
which is Lie-reminiscent is necessarily rigid (see Proposition~\ref{prop:AdmissibleRigid} below). The examples of simple groups provided by Theorems~\ref{thmi:L(G)} and~ \ref{thmi:NewSimpleGroups} below show that the converse fails, even for compactly generated groups in the class $\mS$.

The following result, proved in Theorem~\ref{thm:NonRigid} below, shows in
conjunction with Proposition~\ref{prop:AdmissibleRigid} that any compactly generated simple tree-automorphism group
satisfying Tits' independence property is not Lie-reminiscent.

\begin{theoremintro}\label{thmi:NonRigid}
Let $T$ be a locally finite tree and $G \in \mS$ be a
compactly generated closed subgroup of $\Aut(T)$ satisfying Tits' independence property.
Then $G$ is not rigid.
\end{theoremintro}

The fact that all these groups are not Lie-reminiscent provides a negative answer
to a question of George Willis \cite[Problem 4.3]{Willis07}. Another example of a non-rigid
compactly generated group belonging to $\mS$ was constructed in \cite{BarneaErshovWeigel}.

Our next goal is to give an explicit description of the group of germs of
automorphisms for some groups $G \in \mS$ satisfying Tits' independence
condition. The groups we shall focus on are the \textbf{universal groups
with prescribed local action} defined by
Burger--Mozes~\cite{Burger-Mozes1}. The precise definition of these groups
is recalled in Section~\ref{sec:BM} below.

\begin{theoremintro}\label{thmi:L(G)}
Let $d >1$ and $F \leq \Sym(d+1)$ be a doubly transitive finite permutation group
and  $G = U(F)^+$ be the universal simple group acting on the regular tree
of degree $d+1$ locally like $F$. Then the following conditions are equivalent:
\begin{enumerate}[\rm (i)]
\item $\mathscr L(G)$ is compactly generated.

\item $G$ is locally isomorphic to some compactly generated rigid group $H \in \mS$.

\item $\norma_{\Sym(d)}(F_0) = F_0$, where $F_0$ denotes a point stabiliser in $F$.
\end{enumerate}
If these conditions hold, then the commutator subgroup $[\mathscr L(G), \mathscr L(G)]$ is open,
abstractly simple, and has index~$1$ or $2$ in $\mathscr L(G)$.
\end{theoremintro}

Condition (iii) of Theorem~\ref{thmi:L(G)} clearly excludes the alternating groups or the sharply $2$-transitive groups. Although we did not try to be fully exhaustive here, it seems however that the proportion of doubly transitive finite  permutation groups which satisfy that condition is rather large (see \cite{Cameron} for a list of all finite $2$-transitive groups). 

Theorem~\ref{thmi:L(G)} will be deduced from a detailed study of the group of abstract commensurators of self-replicating profinite wreath branch groups which was largely inspired by a  reading of \cite{BarneaErshovWeigel}. As a consequence of this study, we shall establish the following theorem, which highlights an infinite family of compactly generated locally compact groups which are simple and rigid, but not Lie-reminiscent. We recall that a locally compact group is called \textbf{locally elliptic} if every compact subset is contained in a compact subgroup. We refer to \cite[\S4]{Brown} for a precise definition of the Higman--Thompson group $F_{d, k}$ appearing below. For any $k>0$, the group $F_{2, k}$ is isomorphic to Thompson's group $F$.  

\begin{theoremintro}\label{thmi:NewSimpleGroups}
Let $d > 1$,  $D \leq \Sym(d)$ be transitive and $W = W(D)$ be the profinite branch group defined as the infinitely iterated wreath product of $D$ with itself.  Then for every $k>0$, there is a locally compact group $M = M(D, k)$ which is topologically simple and rigid, and which contains the direct product $W^k$ of $k$ copies of $W$ as a compact open subgroup. Moreover:
\begin{enumerate}[\rm (i)]
\item $M$ is uniquely determined up to isomorphism. 

\item $M$ is compactly generated if and only if  $\norma_{\Sym(d)}(D) = D$. 

\item $[\acomm(M): \comm(M)] \leq 2$.

\item $\acomm(M) =  F_{d, k} \cdot A_k$, where  $F_{d, k}$ is a copy of the Higman--Thompson group  embedded as a discrete subgroup, and $A_k$ is a non-compact locally elliptic open subgroup  such that $F_{d, k} \cap A_k = 1$. 
%Furthermore $F_{d, k}$ is contained in $\comm(M)$ and $A_k$ possesses an open subgroup $A_k^+$ of index at most two which is topologically simple and contained in $\comm(M)$. 
\end{enumerate}
\end{theoremintro}

The connection between Theorems~\ref{thmi:L(G)} and~\ref{thmi:NewSimpleGroups} is the following:
Let $F$, $F_0$ and $G$ be as in Theorem~\ref{thmi:L(G)}, and define $M = M(F_0, 2)$ as in Theorem~\ref{thmi:NewSimpleGroups}.
Then the simple group $[\acomm(G), \acomm(G)]$ appearing in Theorem~\ref{thmi:L(G)} coincides with the group $M = M(F_0, 2)$.
%By Theorem~\ref{thmi:NewSimpleGroups}(v), it does not contain any lattice.

Theorem~\ref{thmi:L(G)} applies notably to the full symmetric group $F = \mbox{$\Sym(d+1)$}$; in
that case $\mathscr L(G)$ coincides with {Neretin's group of
spheromorphisms} introduced in \cite{Neretin}. Equivalently, the Neretin group of the regular tree of degree~$d+1$ is isomorphic to the group $M(\Sym(d), 2)$ appearing in Theorem~\ref{thmi:NewSimpleGroups}. We can thus summarize our results about Neretin's group of spheromorphisms as follows; as mentioned above, the simplicity statement is originally due to Ch.~Kapoudjian~\cite{Kapoudjian}.

\begin{corintro}
	Let $T$ be the regular tree of degree $d+1$, and let $G$ be Neretin's group of spheromorphisms of $T$.
	Then:
	\begin{enumerate}[\rm (i)]
	    \item
		$G = \mathscr L(\Aut(T))$. In particular  $\mathscr L(G) \cong G$, i.e.\@ $G$ is hyperrigid in the sense of \cite{BarneaErshovWeigel}. 
	    \item
		$G$ is compactly generated, abstractly simple and rigid, but not Lie-reminiscent.
	\end{enumerate}
\end{corintro}

\begin{remark*}
	After completing this paper, we learned that Thomas Weigel has generalized Neretin's group
	in a different direction, namely by considering more general classes of rooted trees
	(but always taking the full group of almost automorphisms of such a rooted tree).
	This provides in particular other examples giving a negative answer to George Willis' question
	from \cite[Problem 4.3]{Willis07}; see \cite{Weigel}.
\end{remark*}

\subsection*{Acknowledgements}

We warmly thank Marc Burger for stimulating discussions and for suggesting the term \emph{germ of automorphism}.
We are grateful to Thomas Weigel for helpful comments on an earlier draft of this paper.
Furthermore, two anonymous referees did a wonderful job and provided us with a long list of very useful comments,
which allowed us to improve the exposition of the paper.
A part of this was written up while the first-named author was visiting the Hausdorff Institute for Mathematics in Bonn;
he thanks the institute for its hospitality.

%%%%%%%%%%%%%%%%%%%%%%%%%%%%%%%%%%%%%%%%%%%%%%%%%%%%%%%%%%%%%%%%%%%%%%%%%%%%%%
\section{Preliminaries}\label{sec:prel}
%%%%%%%%%%%%%%%%%%%%%%%%%%%%%%%%%%%%%%%%%%%%%%%%%%%%%%%%%%%%%%%%%%%%%%%%%%%%%%

Let $T$ be a locally finite tree.
We will use the notation $V(T)$ for the set of vertices of $T$, and $E(T)$ for the edge set.
Given $v \in V(T)$, we denote by $E(v)$ the set of edges containing $v$.
Recall that $\Aut(T)$ comes equipped with a natural topology, namely the {\bf permutation topology}, which is defined
by declaring a subgroup $U \leq \Aut(T)$ to be open if and only if $U$ contains the pointwise stabiliser
of some finite subset $S \subset T$.
This topology coincides with the topology of pointwise convergence, as well as with the compact-open topology
(when $V(T)$ is endowed with the discrete topology).
It is Hausdorff, totally disconnected and locally compact;
it is discrete if and only if there is a finite subset $S \subset T$ the pointwise stabiliser of which is trivial.

We first point out that the simple locally compact groups acting properly on $T$ can naturally be viewed as closed subgroups of $\Aut(T)$.

\begin{lemma}\label{lem:metr}
Assume that $G \in \mS$ is a group admitting a continuous and proper action on $T$.
Then $G$ is homeomorphic to its image in $\Aut(T)$, which is closed and metrisable.
\end{lemma}
\begin{proof}
    Since the action is proper and $T$ is locally finite, vertex stabilisers are compact and
    the desired result follows from a standard compactness argument,
    recalling that $G$ is necessarily metrisable and separable (see \cite{KakutaniKodaira} or~\cite[Ch.~II, Th.~8.7]{HewittRoss}).
\end{proof}

From now on, we will only consider closed subgroups of $\Aut(T)$.
We will collect some interesting properties for compactly generated closed subgroups of $\Aut(T)$ contained in $\mS$.
To begin, we give a useful criterion (due to Bass and Lubotzky) for when a closed subgroup of $\Aut(T)$
is compactly generated.
\begin{defn}
    Assume that $G$ is a (closed) subgroup of $\Aut(T)$, such that
    $G$ stabilises a subtree $T' \leq T$ but no subtree of $T'$.
    Then we say that $T'$ is a {\bf minimal invariant subtree} of $T$ for the $G$-action.
\end{defn}
\begin{remark}\label{re:invT}
	\begin{compactenum}[(i)]
	    \item
		Assume that $G$ is a subgroup of $\Aut(T)$
		such that $G$ does not fix an end of $T$.
		Then there exists a minimal $G$-invariant subtree of $T$.
		If in addition $G$ fixes at most one vertex of $T$,
		then this minimal invariant subtree is unique.
		See \cite[Corollaire~3.5]{Tits:trees}.
	    \item
		If $T'$ is a minimal $G$-invariant subtree of $T$,
		then either $T'$ is infinite and has no endpoints (an endpoint is a vertex with valency $1$),
		or $T'$ consists of a single vertex or a single edge.
	\end{compactenum}
\end{remark}

The following fact is well-known; we include a detailed proof for the reader's convenience. 

\begin{lemma}\label{le:ccpt}
    Assume that $G$ is a closed subgroup of $\Aut(T)$,
    and that $T'$ is a minimal $G$-invariant subtree of $T$.
    Then $G$ is compactly generated if and only if its action on $T'$ is cocompact.

    In particular, if $G$ acts edge-transitively on $T$, then $G$ is compactly generated.
\end{lemma}
\begin{proof}
    Observe that the action of $G$ on any locally finite graph $\Gamma$ is cocompact
    if and only if $G$ has finitely many orbits on $\Gamma$.
    (Indeed, a compact subset of $\Gamma$ contains only finitely many vertices and edges.)
%
%    So assume first that the action of $G$ on $T'$ is cocompact;
%    then $G$ has finitely many orbits on $T'$, and hence there are only finitely many cosets of $G_v$ in $G$,
% (WRONG!!)
%    for any vertex $v \in V(T')$, say $G = \bigcup_{i=1}^{k} G_v g_i$.
%    Then $S = G_v \cup \{ g_1,\dots,g_k \}$ is a compact generating set for $G$.
%
    So assume first that the action of $G$ on $T'$ is cocompact;
    then $G$ has finitely many orbits on $T'$.
    Let $C$ be the convex hull of a finite fundamental domain for the action of $G$ on $T'$;
    then $C$ is a finite subtree of $T'$.
    Fix some $v \in C$, and let $H$ be the set of elements of $G$ mapping $v$ into $C$.
    Then $H$ is the union of a finite number of cosets of the compact open subgroup $G_v \leq G$,
    and hence $H$ is a compact subset of $G$.
    Now let $\{ v_1,\dots,v_n \}$ be the sets of vertices adjacent to some vertex of $C$,
    but not contained in $C$.
    Then for each $v_i$, there is an element $g_i \in G$ mapping $v_i$ into $C$.
    Let $S = \{ g_1,\dots,g_n \}$, and fix some $v \in C$;
    then for each $g \in G$, there is some $s \in \langle S \rangle$ mapping
    $v^g$ into $C$ (this follows by induction on the distance from $v^g$ to $C$).
    But then $gs$ maps $v$ into $C$, and hence $gs \in H$.
    We conclude that the compact set $H \cup S$ generates $G$.

    Conversely, assume that $G = \langle S \rangle$ for some compact set $S \subseteq G$.
    Let $v \in V(T')$ be arbitrary;
    then $U := G_v$ is a compact open subgroup of $G$.
    Then $S$ is covered by the open sets $Ug$, where $g$ runs through the elements of~$S$,
    and since $S$ is compact, there is a finite subcover $S \subseteq \{ U g_i \mid 1 \leq i \leq k \}$.
    It follows that $G$ is generated by the compact set $S' = U \cup \{ g_1, \dots, g_k \}$.
    % Since $U$ is compact, there exists a vertex $v \in V(T')$ fixed by $U$.
    For each $i=1,\dots,k$, let $v_i = v \cdot g_i$, and let $F$ be the (finite) convex hull of
    $v, v_1, \dots, v_k$;
    then $F \cup F \cdot s$ is connected, for each $s \in S'$.
    Now let $T''$ be the union of all $G$-translates of $F$.
    Since $G = \langle S' \rangle$, it follows that $T''$ is connected,
    and hence it is a subtree of $T'$;
    it is clearly $G$-invariant, and hence $T'' = T'$ by the minimality of $T'$.
    Since $F$ is finite, this shows that $G$ has only finitely many orbits on $T'$,
    and we conclude that $G$ acts cocompactly on $T'$.
\end{proof}

\begin{defn}
	Recall that an element $g \in \Aut(T)$ acting without inversion is called {\bf elliptic} if it fixes some point,
	and it is called {\bf hyperbolic} otherwise.
	Every hyperbolic element $g \in \Aut(T)$ stabilises a geodesic line, called the {\bf axis} of $g$, on which $g$
	acts by a non-trivial translation.
\end{defn}
Observe that if $G \in \mS$ is a subgroup of $\Aut(T)$, then $G$ acts without inversion on $T$
(else it would contain an index $2$ subgroup).
\begin{lemma}\label{lem:basic}
	Assume that $G \in \mS$ is a closed subgroup of $\Aut(T)$ that is compactly generated.
	Then:
	\begin{enumerate}[\rm (i)]
	    \item
		Every compact subgroup of $G$ is contained in a vertex stabiliser (which is compact and open).
	    \item
		The group $G$ does not fix an end of $T$.
	    \item
		Every ascending chain of compact open subgroups of $G$ stabilises;
		in particular, every compact subgroup of $G$ is contained in a maximal one.
	    \item
		Any open subgroup $H \leq G$ which is not compact contains a hyperbolic element.
	    \item
		If an open subgroup $H \leq G$ fixes an end of $T$, then either $H$ is compact, or $H$ stabilises a geodesic line of $T$.
	    \item
		For every open subgroup $H \leq G$, there exists a minimal $H$-invariant subtree of $T$.
	\end{enumerate}
	Moreover, if the action of $G$ on $T$ is edge-transitive, then:
	\begin{enumerate}[\rm (i)]\setcounter{enumi}{6}
	    \item
		The vertex stabilisers are precisely the maximal compact subgroups of $G$;
		they fall into two conjugacy classes, which correspond to the bipartition of $T$.
	    \item
		Any tree $T'$ admitting a continuous proper edge-transitive
		action of $G$ is isomorphic to $T$.
	\end{enumerate}
\end{lemma}

\begin{proof}
\begin{compactenum}[\rm (i)]
    \item
	Any bounded subset of $T$ admits a canonical centre, hence compact
	subgroups all have a fixed point. Since $G$ acts without inversion,
	a subgroup which fixes a point of $T$ necessarily fixes a vertex.
    \item
	Suppose that $G$ fixes an end $\xi \in \partial(T)$.
	Let $\chi_\xi \colon G \to \ZZ$ be the corresponding Busemann character;
	then for any $g \in G$, $\chi_\xi(g) = 0$ precisely when $g$ is elliptic.
	Since $G$ is simple, the image of $\chi_\xi$ is trivial, hence $G$ consists of elliptic elements only.
	Therefore $G$ stabilises each horoball centred at $\xi$.
	Since the family of all these horoballs is nested,
	it follows that $G$ is the union of an ascending chain of compact subgroups.
	Since $G$ is compactly generated, however, this chain is stabilising, and hence $G$ is compact.
	The desired statement follows since a simple profinite group is necessarily finite, hence it cannot belong to $\mathcal S$.
    \item
	By (ii) and Remark~\ref{re:invT}(i), there exists a minimal $G$-invariant subtree $T'$ of $T$;
	since $G$ is simple, its action on $T'$ is faithful, so there is no loss of generality in
	assuming $T' = T$.
	By Lemma~\ref{le:ccpt}, the action of $G$ on $T$ is now cocompact, \emph{i.e.}\@ $G$ has only finitely many
	orbits on $V(T)$.
	In particular, there are only finitely many conjugacy classes of vertex stabilisers in $G$.

	Every compact open subgroup has finite positive volume;
	since $G$ is simple, it is unimodular, and hence the Haar measure is conjugacy invariant.
	This implies that the volume of any vertex stabiliser is bounded above by a constant;
	since every compact subgroup fixes a vertex by (i),
	it follows that the volume of any compact subgroup is bounded,
	and hence any ascending chain of compact open subgroups stabilises.
    \item
	It is well known that if a group acts on a tree in such a way that every element has a fixed point,
	then the whole group has a fixed point or a fixed end;
	see \cite[Proposition 3.4]{Tits:trees} or \cite[Exercise~2, p.~66]{Serre}.
	Assume that $H \leq G$ is open and contains only elliptic elements.
	If $H$ has a fixed point in $T$, then $H$ is compact and we are done.
	Otherwise $H$ fixes an end $\xi \in \partial T$ and since every element of $H$ is elliptic,
	it follows that $H$ stabilises each horoball centred at $\xi$,
	and hence $H$ is the union of an ascending chain of compact subgroups.
	Since $H$ is open and since every compact subgroup is contained in some compact open subgroup by (i),
	this implies that $G$ contains an ascending chain of compact open subgroups, the union of which is $H$.
	By (iii), this chain is stationary, which implies that $H$ is compact as desired.
    \item
	Let $\xi \in \partial(T)$ be an end fixed by $H$, and assume that $H$ is non-compact.
	Let $\chi_\xi \colon H \to \ZZ$ be the corresponding Busemann character, and let $H_0 = \ker(\chi_\xi)$.
	Then $H_0$ is open, and since it stabilises all horoballs centred at $\xi$, it is a union of compact subgroups;
	hence $H_0$ is a compact open normal subgroup of $H$ by (iii).
	It follows that $H$ acts on the fixed tree $T^{H_0}$ of $H_0$.
	Since $H$ is non-compact, it contains a hyperbolic element by (iv),
	and hence the image of $\chi_\xi$ is non-trivial.
	Let $h \in H$ be such that $\chi_\xi(h)$ generates $\chi_\xi(H)$;
	then $H = \langle H_0, h \rangle$.
	Let $\lambda$ be the axis of the hyperbolic element $h$;
	then $\lambda \subseteq T^{H_0}$ because $\lambda$ is the smallest subtree invariant under $h$.
	We conclude that $H$ stabilises the axis of the hyperbolic element $h$.
    \item
	If $H$ does not fix an end of $T$, then the result follows from Remark~\ref{re:invT}(i).
	So assume that $H$ fixes an end of $T$.
	Then by (v), either $H$ is compact, in which case $H$ fixes a vertex;
	or $H$ stabilises a geodesic line $\lambda$ of $T$,
	so in particular any $H$-invariant subtree contains $\lambda$.
	The intersection of all $H$-invariant subtrees of $T$ is thus non-empty, and clearly minimal. 
    \item
	Assume that some vertex stabiliser $G_v$ is not maximal compact, say $G_v \leq H$ with $H$ compact;
	then by (i), there is another vertex $w$ such that $G_v \leq G_w$, and
	by considering the path from $w$ to $v$, we may assume that $w$ is a neighbour of $v$.
	This would imply that every element of $G$ fixing $v$ would fix the edge $\{v,w\}$,
	but this contradicts the edge-transitivity.

	Since $G$ acts edge-transitively and without inversion, it has precisely two orbits on the vertices of $T$;
	the second claim follows.
    \item
	Since the $G$-action is edge-transitive, the tree $T$ is biregular.
	It suffices to show that the two valencies of the two classes of vertices are uniquely determined by $G$.
	Denoting by $\mathcal U_1$ and $\mathcal U_2$ the two conjugacy classes of maximal compact subgroups of $G$ provided by~(vii),
	one verifies easily that these two degrees coincide with
	\[ \min \{ |U_1: U_1 \cap U_2|,  \; U_1 \in \mathcal U_1,  \; U_2 \in \mathcal U_2\} \phantom{\,.} \]
	and
	\[ \min \{ |U_2: U_1 \cap U_2|,  \; U_1 \in \mathcal U_1,  \; U_2 \in \mathcal U_2\} \,, \]
	respectively.
    \qedhere
\end{compactenum}
\end{proof}

%\begin{remark}
%Let $G \in \mS$ be a group admitting a continuous and proper edge-transitive action on a locally finite tree $T$. By Lemma~\ref{lem:basic}(viii), the isomorphism class of the tree $T$ is uniquely determined by $G$. We do not know, however, whether the $G$-action on $T$ is uniquely determined; in other words, is it possible that $\Aut(T)$ contains two isomorphic closed subgroups which are simple and edge-transitive, but not conjugate?
%\end{remark}

%%%%%%%%%%%%%%%%%%%%%%%%%%%%%%%%%%%%%%%%%%%%%%%%%%%%%%%%%%%%%%%%%%%%%%%%%%%%%%
\section{Groups with Tits' independence property}\label{sec:Tits}
%%%%%%%%%%%%%%%%%%%%%%%%%%%%%%%%%%%%%%%%%%%%%%%%%%%%%%%%%%%%%%%%%%%%%%%%%%%%%%

\subsection{Tits' independence property}\label{ss:Tits}

The following remarkable theorem of Tits \cite{Tits:trees} shows that if $G$ is in some sense
large enough, then $G$ has a (usually rather big) abstractly simple subgroup.
We first make our definition of ``large enough'' precise.
\begin{defn}
	Let $G \leq \Aut(T)$, let $C$ be a (finite or infinite) chain of $T$, and let $F$ be the pointwise
	stabilizer of $C$ in $G$.
	For each vertex $v$ of $T$, we denote by $\pi(v)$ the vertex of $C$ closest to $v$.
	The vertex sets $\pi^{-1}(c)$ with $c \in V(C)$ are all invariant under $F$; let $F_c$ denote
	the permutation group obtained by restricting the action of $F$ to $\pi^{-1}(c)$.
	Then there is a natural homomorphism
	\[ \varphi \colon F \to \prod_{c \in V(C)} F_c \,. \]
	We say that $G$ satisfies {\bf Tits' independence property}
	if and only if $\varphi$ is an isomorphism.
\end{defn}
\begin{remark}
	If $G \leq \Aut(T)$ is a {\em closed} subgroup, then $G$ satisfies Tits' independence property
	if and only if for every edge $e \in E(T)$, the pointwise edge stabiliser $G_{(e)}$ can be decomposed as
	$G_{(e)} = G_{(h_1)} G_{(h_2)}$, where $h_1$ and $h_2$ are the rooted half-trees emanating from the
	endpoints of $e$, \emph{i.e.}\@ $E(T)$ is the disjoint union of $E(h_1)$, $e$ and $E(h_2)$.
	See, for example, \cite[Lemma~10]{Am}.
\end{remark}
\begin{theorem}[\cite{Tits:trees}]\label{th:tits}
	Let $T$ be a simplicial tree, $G$ a subgroup of $\Aut(T)$, and $G^+$ the subgroup of $G$ generated
	by the edge-stabilizers.
	Assume that $G$ does not stabilize a proper subtree, and that $G$ does not fix an end of $T$.
	If $G$ satisfies Tits' independence property, then $G^+$ is simple (or trivial).
\end{theorem}

\subsection{Going from $G$ to $G^+$}

As we will see, the group $G^+$ can be characterised topologically as the \textbf{monolith} of $G$, \emph{i.e.} the unique minimal closed normal subgroup of $G$ or, equivalently, the intersection of all non-trivial closed normal subgroups of $G$. We first recall the definition of the quasi-centre, which was introduced by Burger and Mozes in \cite{Burger-Mozes1}.
\begin{defn}\label{def:QZ}
	The {\bf quasi-centre} $\QZ(G)$ of $G$ is
	the characteristic (but not necessarily closed) subgroup
	consisting of all those elements admitting an open centraliser.
	(This group is sometimes called the {\em virtual centre} $\mathrm{VZ}(G)$ by some authors.)
\end{defn}
The following elementary but important fact is well known, and we will use it frequently without explicitly mentioning it.
\begin{lemma}
	If $N \unlhd G$ is a discrete normal subgroup, then $N \leq \QZ(G)$.
	In particular, if $\QZ(G) = 1$, then every non-trivial normal subgroup is non-discrete.
\end{lemma}
\begin{proof}
	Let $N \unlhd G$ be a discrete normal subgroup.
	Let $g \in N$ be arbitrary; then $g^G$ is a discrete conjugacy class.
	This implies that $g$ has an open centraliser, and hence $g \in \QZ(G)$.
\end{proof}
\begin{prop}
	Let $G$ be a closed subgroup of $\Aut(T)$ satisfying Tits' independence property,
	and assume that $G$ does not stabilize a proper subtree, and that $G$ does not fix an end of $T$.
	Let $G^+$ be the subgroup generated by its edge stabilisers.
	Then
	\begin{compactenum}[\rm (i)]
	    \item
		$\QZ(G) = 1$.
	    \item
		$G^+$ is the unique minimal closed normal subgroup of $G$.
	\end{compactenum}
\end{prop}
\begin{proof}
	\begin{compactenum}[(i)]
	    \item
		Let $g \in \QZ(G)$, and let $v$ be an arbitrary vertex of $T$.
		Then by definition, there exists a finite subset $S \subset T$ such that $g$ centralises the pointwise stabiliser $G_{(S)}$,
		and we may assume that $S$ contains the vertex $v$ (by replacing $S$ by $S \cup \{ v \}$).
		Let $\tilde{S}$ be the fixtree of $G_{(S)}$; in particular, $G_{(S)} = G_{(\tilde{S})}$.
		Then $g$ stabilises $\tilde{S}$. By Tits' independence property the group $ G_{(\tilde{S})}$ splits
		as a direct product of many factors, and each factor naturally corresponds to a vertex of $\tilde S$
		contained in at least one edge not in $\tilde S$.
		Since $g$ centralises $G_{(\tilde{S})}$,
		it normalises each of these factors and, hence, it acts trivially on the vertices of $\tilde{S}$
		contained in at least one edge not in $\tilde S$; but then it acts trivially on all the vertices of $\tilde{S}$.
		In particular, $g$ fixes $v$, and since $v$ was arbitrary, $g$ is trivial, as desired.
		
		% In particular, we have shown that $g$ belongs to any open subgroup that it centralises.
		% Since $g$ centralises every open subgroup of $G_{(S)}$, we conclude that $g$ is trivial, as desired.
	    \item
		Let $1 \neq N$ be a closed normal subgroup of $G$;
		we will show that $N$ contains $G^+$.
		If $N \cap G^+ = 1$, then $N$ would be a discrete normal subgroup of $G$.
		This is impossible, however, since $\QZ(G) = 1$ by (i).
		% (Indeed, every element with a discrete conjugacy class has an open centraliser.)
		Hence $N \cap G^+ = G^+$ since $G^+$ is simple by Theorem~\ref{th:tits},
		and this proves our claim.
	    \qedhere
	\end{compactenum}
\end{proof}

We point out that $G^+$ need not be compactly generated, even if the group $G$ is.

\begin{example}
Let $T$ be the Cayley tree of the free group $F_2$ associated with a basis $\{a,b \}$. We view $T$ as a bi-coloured graph (one colour per generator). Let $G \leq \Aut(T)$ be the full automorphism group of this coloured graph. Then $G$ is closed; it is furthermore cocompact since it contains $F_2$.
Let $G^+$ denote the (open) subgroup of $G$ generated by the pointwise edge-stabilisers.
Then $G^+$ is simple but not compactly generated.
\end{example}
\begin{proof}
The fact that $G^+$ is simple follows from Theorem~\ref{th:tits}.

Let $V = F_2$ be the vertex-set of $T$.
Let $D_\infty = \langle a', b' \mid (a')^2 = (b')^2 = 1 \rangle$ denote the infinite dihedral group.
Let also $f \colon F_2 \to D_\infty$ be the canonical homomorphism defined by $f(a)= a'$ and $f(b)=b'$.
We now define a map
\[ \delta \colon V \times V \to D_\infty \colon (v, w) \mapsto f(w^{-1}v) \,. \]
Observe that $\delta(v,w)$ only depends on the colours of the edges of the unique path from $v$ to $w$;
in particular, $\delta$ is $G$-equivariant. The crucial observation is now the following: for each $g \in G$ which fixes pointwise some edge of $T$, we have $\delta(v, g.v) = 1$ for all $v \in V$. Since $G^+$ is generated by edge-stabilizers, it readily follows that $\delta(v, g.v) = 1$ for all $g \in G^+$ and $v \in V$.
Now pick a base vertex $v_0$ and consider an `infinite staircase' consisting of a bi-infinite geodesic line $(\dots, v_{-1}, v_0, v_1, \dots)$ such that
$\delta(v_0, v_{2n}) = (a' b')^n$ for every integer $n$.
It follows from the preceding discussion that no element of $G^+$ can map $v_0$ to $v_i$ for any $i \neq 0$;
in fact, this bi-infinite geodesic line is a fundamental domain for the $G^+$-action on $T$.
In particular, we conclude that the $G^+$-action is not cocompact.
Since $G^+$ acts minimally on $T$, Lemma~\ref{le:ccpt} implies that $G^+$ is not compactly generated;
on the other hand, $G$ does act cocompactly on $T$ and hence the same lemma implies that $G$ is compactly generated.
\end{proof}

% It is not difficult to show that an `infinite staircase' consisting of a bi-infinite geodesic line of $T$ whose pairs of successive edges have different colours, is in fact a fundamental domain for the $G^+$-action on $T$.

\subsection{Open subgroups}

It turns out that for simple groups with Tits' independence property, we have some control over the open subgroups.
\begin{prop}\label{prop:Titsopen}
	Let $G \in \mS$ be a closed subgroup of $\Aut(T)$ satisfying Tits' independence property,
	and let $H$ be an open subgroup of $G$.
	If $T'$ is a minimal $H$-invariant subtree of $T$ (which always exists by Lemma~\textup{\ref{lem:basic}(vi)}),
	then $H$ contains the full edge-stabiliser $G_e$ for each edge $e \in E(T')$.
\end{prop}
\begin{proof}
	If $H$ is compact, then it is contained in a vertex stabiliser $G_v$, and hence
	every minimal $H$-invariant subtree of $T$ is reduced to a single vertex;
	in this case, the statement is empty.

	So assume that $H$ is not compact.
	Then by Lemma~\ref{lem:basic}(iv), there is some hyperbolic element $\gamma \in H$;
	let $L$ be the axis of $\gamma$.
	Since $H$ is open, on the other hand, there is some finite subtree $S$ of $T$ such that the pointwise
	stabiliser $G_{(S)}$ of $S$ is contained in $H$.

	We claim that $L \subseteq T'$.
	Indeed, let $v \in V(T')$ be arbitrary, and let $z$ be the vertex of $L$ closest to $v$;
	then $\gamma$ maps $z$ to another vertex $z\gamma$ of $L$, and the path between $v$ and $v\gamma$
	contains the path between $z$ and $z\gamma$, which contains at least one edge $e$ of $L$.
	Since $v, v\gamma \in V(T')$, we conclude that $e \in E(T')$, and hence the $\langle \gamma \rangle$-orbit
	of $e$ is contained in $E(T')$; since the convex hull of this orbit is precisely $L$, we have $L \subseteq T'$ as claimed.

	Next, suppose that $e$ is an edge of $T$, and let $T_1$ and $T_2$ be two subtrees of~$T$
	not containing $e$, and lying at different sides of~$e$.
	We claim that if the pointwise stabilisers $G_{(T_1)}$ and $G_{(T_2)}$ are both contained in~$H$,
	then so is the edge stabiliser $G_e$.
	Indeed, let $T'_i$ be the unique half-tree rooted at one of the endpoints of $e$ containing $T_i$ for $i \in \{ 1,2 \}$.
	Then $G_{(T'_i)} \leq G_{(T_i)} \leq H$ for each $i$, and hence, by Tits' independence property,
	\[ G_e = \bigl\langle G_{(T'_1)} \cup G_{(T'_2)} \bigr\rangle \leq H \]
	as claimed.
	In particular, if $e'$ and $e''$ are two edges of $T$ such that
	$G_{e'}$ and $G_{e''}$ are contained in $H$, then for every edge $e$ lying
	on the geodesic from $e'$ to $e''$ we have $G_e \leq H$ as well.

	Now let $e \in E(L)$ be arbitrary.
	By applying appropriately high positive or negative powers of $\gamma$ to $S$, we get subtrees $S_1$ and $S_2$
	not containing $e$ and lying at different sides of $e$.
	Since $G_{(S_i)} \leq \langle G_{(S)}, \gamma \rangle \leq H$ for both $i=1,2$,
	we can apply the previous paragraph to conclude that $G_e \leq H$.

	Since $T'$ is a minimal $H$-invariant subtree of $T$ containing the edge $e$, the convex hull
	of the $H$-orbit of $e$ (which is clearly $H$-invariant) has to coincide with $T'$ itself.
	So let $f \in E(T')$ be arbitrary; then there exist elements $\alpha_1,\alpha_2 \in H$ such that
	$f$ lies on the geodesic from $e\alpha_1$ to $e\alpha_2$.
	Since $G_{e\alpha_i} = G_e^{\alpha_i} \leq H$, we conclude that $G_f \leq H$ as well.
\end{proof}

We say that a locally compact group has {\bf few open subgroups} if every proper open subgroup
is compact.
The large edge-transitive tree automorphism groups with few open subgroups can be characterised in terms of their local action.

\begin{theorem}\label{thm:LocPrim}
Let $T$ be a tree and $G \in \mS$ be a compactly generated non-compact closed subgroup of $\Aut(T)$ satisfying Tits' independence property. Then the following conditions are equivalent.
\begin{enumerate}[\rm (i)]
\item Every proper open subgroup of $G$ is compact.

\item $G$ splits as an amalgamated free product $G \cong A *_C B$, where $A$ and $B$ are maximal compact open subgroups and $C = A \cap B$ is a maximal subgroup of both $A$ and $B$. Moreover the Bass--Serre tree associated to this amalgam admits a $G$-equivariant embedding in $T$ which is isometric up to a scaling factor. 

\item $G$ splits as an amalgamated free product $G \cong A *_C B$, where $A$ and $B$ are maximal compact open subgroups,  $C = A \cap B$ and the $A$-action on $A/C$ (resp. the $B$-action on $B/C$) is primitive and non-cyclic. Moreover the Bass--Serre tree associated to this amalgam admits a $G$-equivariant embedding in $T$ which is isometric up to a scaling factor. 
\end{enumerate}
\end{theorem}

A tree is called \textbf{thick} if the valence of every vertex is at least three. The following immediate corollary is nothing but a reformulation of Theorem~\ref{thmi:LocPrim} from the introduction. 

\begin{cor}\label{cor:LocPrim}
Let $T$ be a thick tree and $G \in \mS$ be a compactly generated closed subgroup of $\Aut(T)$ acting minimally on $T$ and satisfying Tits' independence property. Then the following conditions are equivalent.
\begin{enumerate}[\rm (i)]
\item Every proper open subgroup of $G$ is compact.

\item For every vertex $v \in V(T)$, 	the induced action of $G_v$ on $E(v)$ is primitive; in particular $G$ is edge-transitive.

\item For every vertex $v \in V(T)$,
	the induced action of $G_v$ on $E(v)$ is primitive and non-cyclic; in particular $G$ is edge-transitive.

\end{enumerate}
\end{cor}

We shall need a subsidiary observation, which clarifies the equivalence between (ii) and (iii) in Theorem~\ref{thm:LocPrim} and Corollary~\ref{cor:LocPrim}.

\begin{lemma}\label{lem:LocCyclic}
Let $G\leq \Aut(T)$  be a closed subgroup  satisfying Tits' independence property, acting edge-transitively on $T$.  If $G \in \mS$, then there is no vertex $v \in V(T)$ such that the $G_v$-action on $E(v)$ is free.
\end{lemma}

\begin{proof}
%Suppose $G$ is not trivial. Then it is non-compact (because $G$ is simple) and hence $T$ has no endpoint. 
%
Since $G$ is edge-transitive, it follows that $G_v$ is transitive on $E(v)$ for all $v \in V(T)$. Let now $v_0 \in V(T)$ be such that  the $G_{v_0}$-action on $E(v_0)$ is free. Let $v_0^\perp$ denote the subset of $V(T)$ consisting of all vertices adjacent to, but distinct from, $v_0$. Let $G_0$ denote the subgroup of $G$ generated by $\{G_w \; | \; w \in v_0^\perp\}$. Then $G_0$ is an open subgroup of $G$ and moreover we have $E(v_0) \cdot G_0 = E(T)$. Notice moreover that, in view of the assumption that $G_{v_0}$ acts freely on $E(v_0)$, it follows from Bass--Serre theory that $E(v_0)$ is a strict fundamental domain for the $G_0$-action on $E(T)$ and that $G_0$ is the free product of the groups $\{G_w \; | \; w \in v_0^\perp\}$ amalgamated over their intersection $\bigcap_{ w \in v_0^\perp} G_w $. 

Since $G_0$ acts cocompactly on $T$, we infer that $T$ is a minimal invariant subtree for $G_0$. In view of Proposition~\ref{prop:Titsopen}, it follows that $G_e$ is contained in $G_0$ for every edge $e \in E(T)$. Since $G$ is simple, it is generated by its edge-stabilisers, and we infer that $G = G_0$. Since $G$ acts edge-transitively while $G_0$ has exactly $| E(v_0)|$ orbits on $E(T)$, we infer that the star $E(v_0)$ is in fact reduced to a single edge of $T$. Therefore $v_0$ is an endpoint of $T$. In particular $T$ is bounded and hence, $G$ is finite. This is absurd since the class $\mS$ does not contain any discrete group.
\end{proof}

\begin{proof}[Proof of Theorem~\ref{thm:LocPrim}]
Observe that by Proposition~\ref{prop:Titsopen}, each open subgroup $H$ of $G$ has a minimal invariant subtree, and such a subtree reduces to a single vertex if and only if $H$ is compact. There is no loss of generality in assuming that $G$ acts minimally on $T$. Since $G$ is compactly generated, we deduce from Lemma~\ref{le:ccpt} that the action of $G$ on $T$ is cocompact. 

\begin{compactitem}
\item[(i) $\Rightarrow$ (ii)]

Bass--Serre theory provides us with a graph of group decomposition of $G$ over a finite graph $\Gamma$. The corresponding vertex- and edge-groups are nothing but vertex- and edge-stabilisers for the $G$-action on $T$. In particular they are compact open subgroups of $G$. 
Collapsing the aforementioned graph of group decomposition to a single edge, we obtain a presentation of $G$ as a non-trivial amalgamated free product or as an HNN-extension. Since $G$ is simple, it does not map onto $\mathbf Z$ and can therefore not be an HNN-extension. Thus we have a decomposition $G = A *_C B$ with $C= A \cap B$. Since $A$ and $B$ are generated by vertex-groups of the initial decomposition, they are open subgroups of $G$. They must therefore be compact since $G$ has few open subgroups. 

Let $\tilde T$ be the Bass--Serre tree associated with the decomposition  $G = A *_C B$, and let $\tilde x, \tilde y \in V(\tilde T)$ denote the vertices respectively fixed by $A$ and $B$. Applying Lemma~\ref{lem:basic} to $\tilde T$, we deduce that $A$ and $B$ are maximal compact subgroups. In particular $A$ and $B$ are vertex stabilisers for the $G$-action on $T$. 

\medskip    
We now show that the $G$-action on $\tilde T$ is locally primitive. To this end, suppose that there is a vertex $v \in V(\tilde T)$ for which the local action of $G_v$ on $E(v)$ is not primitive. Fix an edge $e = \{ v, w \} \in E(\tilde T)$. Since $G$ is edge-transitive, it follows that the $G_v$-action on $E(v)$ is transitive. However it is not primitive by assumption; we deduce that there is a subgroup $H_1 \leq G_v$ containing $G_e$ properly but which does not act transitively on $E(v)$. In particular $H_1 \neq G_v$.

Set $H_2 = G_w$ and consider the open subgroup $H = \langle H_1, H_2 \rangle$. In fact, we have $H = H_1 *_{G_e} H_2$ and the set $S = e \cdot H$ is a minimal $H$-invariant subtree of $\tilde T$, which is a Bass--Serre tree corresponding to the above decomposition of $H$. Since  $H_1$ is not transitive on $E(v)$, we infer that $S$ is a proper subtree of $\tilde T$. In particular $H$ is properly contained in $G$. By construction the tree $S$ is unbounded, so that $H$ is non-compact. Thus we have shown that $G$ contains a proper non-compact open subgroup, which is absurd. This confirms that $G$ acts locally primitively on $\tilde T$ as claimed.

\medskip
It only remains to exhibit a $G$-equivariant embedding $\tilde T \to T$. To this end, we pick vertices $x, y$ of $T$ with $A = G_x$ and $B = G_y$ and such that  $d(x, y)$ is minimal with respect to this property. Since $[x, y] \cdot G$ is a connected and $G$-invariant subset of $T$, it coincides with $T$ by minimality of the $G$-action. Since the $G$-action on $\tilde T$ is locally primitive, it easily follows that $A, B$ and $C$ are the only compact subgroups of $G$ which contain $C$ as a subgroup. Now for every element $h \in G$ such that $zh \in [x, y]$ for some $z \in ]x, y[$, the group $C = G_x \cap G_y$ is contained in the compact subgroup $G_z \cap h\inv G_z h$. This implies that $C = G_z$ and, in particular, that $z$ fixes $[x, y]$ pointwise. In other words, this shows that $[x, y] $ is a strict fundamental domain for the $G$-action on $T$. On the other hand it is clear that the edge $[\tilde x, \tilde y]$ is a strict fundamental domain for the $G$-action on $T$. Therefore, the assignments $\tilde x \mapsto x$ and $\tilde y \mapsto y$ extend to a $G$-equivariant map $\tilde  T \to T$ which is isometric up to multiplying the metric on $\tilde T$ by a factor equal to $d(x, y)$, the presence of the latter factor accounting for the possibility that the segment $[x, y]$ be divided into several edges by some vertices of valence two.

\item[(ii) $\Rightarrow$ (iii)]
	Immediate from Lemma~\ref{lem:LocCyclic}.

    \item[(iii) $\Rightarrow$ (i)]
We may and shall assume without loss of generality that $T$ coincides with the Bass--Serre tree of the given amalgam decomposition of $G$. 
Assume that for every vertex $v \in V(T)$, the induced action of $G_v$ on $E(v)$ is primitive and non-cyclic. For every pair of edges $e_1, e_2$ containing $v$, we have that both $G_{e_1}$ and $G_{e_2}$ are two different maximal subgroups of $G_v$, which are non-trivial since the $G_v$-action on $E(v)$ is not cyclic. We deduce that 	$G_v = \langle G_{e_1}, G_{e_2} \rangle$.
		
	Now let $H$ be an arbitrary open subgroup of $G$, and suppose that $H$ is not compact. Then there is a minimal invariant subtree $S$ for $H$, which has no endpoints.
	By Proposition~\ref{prop:Titsopen}(ii), $H$ contains every edge stabiliser $G_e$ with $e \in E(S)$.
	Since $G$ is generated by the edge stabilisers and $H \neq G$, the equality $H=G$ will follow provided we show that $S=T$. 
		
	Since $S$ has no endpoint, it follows that for every vertex $v \in V(S)$, the star $E(v)$ contains at least two edges $e_1, e_2$ in $S$. By the previous paragraph, we have $G_v = \langle G_{e_1}, G_{e_2} \rangle \leq H$. Since $G_v$ acts transitively on $E(v)$, we deduce that for every $v \in V(S)$ we have $S_v = E(v)$. Clearly this implies that $S = T$, as desired.
\qedhere
\end{compactitem}
\end{proof}

%%%%%%%%%%%%%%%%%%%%%%%%%%%%%%%%%%%%%%%%%%%%%%%%%%%%%%%%%%%%%%%%%%%%%%%%%%%%%%
\section{Groups with a prescribed local action}
%%%%%%%%%%%%%%%%%%%%%%%%%%%%%%%%%%%%%%%%%%%%%%%%%%%%%%%%%%%%%%%%%%%%%%%%%%%%%%

\subsection{Burger-Mozes' universal group $U(F)$}\label{sec:BM}

We will now focus on a family of examples constructed by M.~Burger and Sh.~Mozes~\cite{Burger-Mozes1}.

Let $d>2$ be a positive integer, let $F \leq \Sym(d)$ be a permutation group on the set $\mathbf d = \{1, \dots, d\}$ and let $T$ be the regular tree of degree $d$. Pick a colouring $i \colon E(T) \to \mathbf d$ of the edge-set of $T$ by the elements of $\mathbf d$ such that its restriction to the star $E(v)$ around every vertex $v$ is a bijection. It is clear that this colouring is unique up to an automorphism of $T$. Let now $U(F)$ be the automorphism group defined by
\[ U(F) = \{g \in \Aut(T) \mid  i \circ g  \circ (i |_{E(v)})^{-1} \in F \text{ for all  } v \in V(T) \} \,. \]
Let $U(F)^+$ denote the subgroup generated by the pointwise edge-stabilisers. One shows  that $U(F)^+$ is edge-transitive if and only if $F$ is transitive and generated by its point stabilisers. In that case $U(F)^+$ has index two in $U(F)$ and it follows from Theorem~\ref{th:tits} that $U(F)^+$ is simple. Furthermore, this assumption also ensures that the group $G = U(F)^+$ acts locally as~$F$; in other words for each vertex $v \in V(T)$ the $G_v$-action on the star $E(v)$ is isomorphic to the $F$-action on $\mathbf d$. It is shown in~\cite[\S3.2]{Burger-Mozes1} that every vertex-transitive subgroup of $\Aut(T)$ whose vertex stabilisers act locally like $F$ (on the star of the fixed vertex) is conjugate to a subgroup of $U(F)$.

\begin{quote}\em
    We will assume from now on that $F$ is transitive and generated by its point stabilisers.
\end{quote}

\subsection{Open subgroups of $U(F)^+$}

We retain the notation and assumptions of Section~\ref{sec:BM}. Let us moreover denote by $G$ the simple group $U(F)^+$.

\begin{prop}\label{pr:openU(F)}
We have the following.
\begin{enumerate}[\rm (i)]
    \item
	The group $F$ is primitive if and only if every proper open subgroup of $G$ is compact.
    \item
	Suppose that $F$ is imprimitive, with maximal blocks of imprimitivity of cardinality~$k$,
	and assume moreover that $F$ acts regularly on each such block.
	Then for each edge $e \in E(T)$, the quotient $\norma_G(G_e)/G_e$ is virtually free.
	Moreover, if $k \geq 3$, then $G$ possesses open subgroups which are not compactly generated.
\end{enumerate}
\end{prop}
\begin{proof}
\begin{compactenum}[\rm (i)]
    \item
	This follows immediately from Theorem~\ref{thm:LocPrim}.
    \item
	We shall use the following definition.    For each subset of colours $\mathbf{b} \subseteq \mathbf{d}$, we define a {\bf $\mathbf{b}$-tree}
    to be a subtree $S \subseteq T$ which only uses colours from $\mathbf{b}$, and which is maximal
    with respect to this property, \emph{i.e.}\@ for each vertex $v \in V(S)$ we have $i(S \cap E(v)) = \mathbf{b}$.
   
Let now $e$ be an arbitrary edge of $T$, and let $\mathbf{b} \subseteq \mathbf{d}$ the unique block of imprimitivity
	containing the colour of $e$.
	Denote the $\mathbf{b}$-tree containing $e$ by $S$;
	then $S$ is a regular tree of degree $|\mathbf{b}| = k$.
	By the regularity condition imposed on $F$, every element of $G_e$ fixes the tree $S$ elementwise,
	and in fact $G_e = G_f$ for every $f \in E(S)$.
	On the other hand, the maximality of the blocks implies that a point stabiliser $F_a$ does not fix any point
	outside the block containing $a$, and this implies that $G_e \neq G_f$ for every $f \not\in E(S)$.
	We conclude that $\norma_G(G_e)$ is equal to the global stabiliser $G_S$ of the subtree $S$, and clearly
	$\norma_G(G_e)$ acts edge-transitively on $S$.
	Hence the discrete group $\norma_G(G_e) / G_e$ acts transitively and properly (\emph{i.e.}\@ with finite vertex stabilisers) on $S$, and
	Bass--Serre theory implies that $\norma_G(G_e) / G_e$ is virtually free.
	Moreover, it is virtually abelian free if and only if the tree $S$ is a line or a point, \emph{i.e.}\@ if $k < 3$. In particular, if $k\geq 3$,  then $\norma_G(G_e) / G_e$ contains subgroups that are not finitely generated.
	Lifting such a group back to a subgroup of $\norma_G(G_e)$ provides us with an open subgroup of $G$ which is
	not compactly generated.
    \qedhere
\end{compactenum}
\end{proof}
\begin{remark}
A topological group is called \textbf{Noetherian} if it satisfies the ascending chain condition on open subgroups. A locally compact group is Noetherian if and only if every open subgroup is compactly generated. In particular, it follows from Proposition~\ref{pr:openU(F)} that compactly generated elements of $\mS$ need not be Noetherian. %Notice also that chains of non-compact open subgroups may have arbitrarily large finite length even in the Noetherian case.
\end{remark}

\subsection{Local structure}\label{sec:LocalStr}

By the {\bf local structure} of $G$, we mean the properties shared by all compact open subgroups of $G$; such properties are invariant up to commensurability.

The structure of vertex-stabilisers in $G$ may be described in terms of the finite group $F$.
More precisely, the vertex-stabilisers have the structure of an infinitely iterated wreath product of finite groups:
\begin{prop}[{\cite[Section~3.2]{Burger-Mozes1}}]\label{pr:wr}
    Let $F \leq \Sym(d)$.
    The maximal compact subgroup $U(F)^+_v$ of $U(F)^+$
    is obtained as the projective limit $\varprojlim A_n$ of the
    groups
    \[ \begin{cases}
        A_0 = F \,; \\
        A_n = F_a \wr A_{n-1} & \text{for } n \geq 1,
    \end{cases} \]
    where $F_a  \leq  F$ denotes a point-stabiliser. Similarly, for an edge stabiliser $U(F)^+_e$ we have
    $U(F)^+_e \cong \varprojlim D_n \times \varprojlim D_n \cong \varprojlim (D_n \times D_n)$ with
    \[ \begin{cases}
        D_1 = F_a \,; \\
        D_n = F_a \wr D_{n-1} = D_{n-1} \wr F_a & \text{for } n \geq 2 \,.
    \end{cases} \]
\end{prop}

We warn the reader that all wreath products are considered with their imprimitive wreath product action,
and that the point stabilisers $F_a$ are considered with their action on $(d-1)$ elements
(even if this action has fixed points). We shall come back to the profinite group $D = \varprojlim D_n$ in Section~\ref{sec:branch} below.

We will now investigate the extent to which the group $F$ is determined by the local or global structure of $G$. 

\begin{theorem}\label{th:global}
\begin{compactenum}[\rm (i)]
    \item
    Let $F, F'  \leq  \Sym(d)$ be two groups satisfying the conditions of Section~\textup{\ref{sec:BM}},
    and assume that $U(F)^+ \cong U(F')^+$.
    Then $F \cong F'$ as permutation groups, \emph{i.e.}\@ $F$ and $F'$ induce equivalent permutation representations on $\mathbf{d}$.
    \item
    Let $G' \in \mathcal{S}$ be a group acting continuously, properly and edge-transitively on some tree $T$,
    and assume that $G' \cong G = U(F)^+$ for some group $F  \leq  \Sym(d)$ satisfying the conditions of Section~\textup{\ref{sec:BM}}.
    Then $G$ and $G'$ have equivalent actions on the tree $T$ of degree $d = |\mathbf{d}|$.
\end{compactenum}
\end{theorem}
\begin{proof}
    Let $F  \leq  \Sym(d)$ be a group satisfying the conditions of Section~\ref{sec:BM}.
    We will show that we can recover $F$ from the group $G = U(F)^+$ alone, and that the local action
    around every vertex of the tree is given by the action of $F$ on $\mathbf{d}$;
    this will simultaneously prove (i) and (ii).

    By Lemma~\ref{lem:basic}(viii), we know that the number $d$, which is the degree of the tree $T$ on which $G$ acts,
    is uniquely determined by $G$.
    Let $K$ be an arbitrary maximal compact subgroup of $G$, and let $L$ be another
    maximal compact subgroup of $G$ minimizing the index $[K : K \cap L]$.
    Define 
    \[ C = \bigcap_{g \in K} g\inv (K \cap L) g. \]
    We claim that $K/C \cong F$ and that the conjugation action of $K/C$ on
    the set of conjugates of $K \cap L$ in $K$
    is equivalent to the permutation action of $F$ on~$\mathbf{d}$.

    By Lemma~\ref{lem:basic}(vii), $K$ and $L$ are two (different)
    vertex stabilisers, say $K = G_v$ and $L = G_w$ with $v,w \in V(T)$.
    Let $z$ be the unique neighbour of $v$ lying on the unique path from $v$ to $w$;
    then $G_v \cap G_w \leq G_v \cap G_z$, so by the minimality of the index
    $[K : K \cap L]$, we have $K \cap L = G_v \cap G_z$.
    Since $K = G_v$ acts transitively on the star $E(v)$, we see that
    $C$ is equal to the pointwise stabiliser of $E(v)$, \emph{i.e.}\@ it is the kernel
    of the action of $K$ on $E(v)$.
    We conclude that $K/C \cong F$, and the action of $K/C$ on $E(v)$ is precisely
    given by the action of $F$ on $\mathbf{d}$.
\end{proof}

%% (UNNATURAL QUESTION FROM THE TOPOLOGICAL POINT OF VIEW)
%%
%% At this point, we do not yet know the answer to the similar question starting from the local structure of the groups;
%% in particular, we do not know whether two non-isomorphic groups $U(F)^+ \not\cong U(F')^+$ can nevertheless have
%% isomorphic maximal compact subgroups $U(F)^+_v \cong U(F')^+_v$.
%% \begin{question}
%%     Let $F, F'  \leq  \Sym(d)$ be two groups satisfying the conditions of Section~\ref{sec:BM},
%%     and assume that $U(F)^+_v \cong U(F')^+_v$.
%%     Is it true that $F \cong F'$ as permutation groups?
%% \end{question}

We can summarise the relations between the structure of $F$ (as a permutation group), $U(F)^+$ (as a topological group),
and $U(F)^+_v$ (as the commensurability class of the profinite group $U(F)^+_v$) in the following diagram.

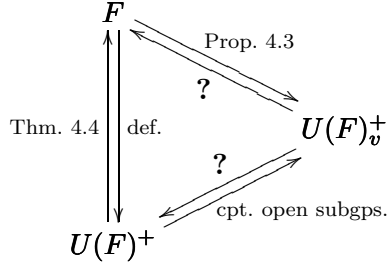
\[ \bfig
  \node a(800,400)[U(F)^+_v]
  \node b(0,800)[F]
  \node c(0,0)[U(F)^+]
  \arrow|l|/@<2pt>/[c`b;\text{Thm.~\ref{th:global}}]
  \arrow|r|/@<2pt>/[b`c;\text{def.}]
  \arrow|l|/@<2pt>/[a`b;\text{\bf\normalsize ?}]
  \arrow|r|/@<2pt>/[b`a;\text{\ Prop.~\ref{pr:wr}}]
  \arrow|l|/@<-2pt>/[a`c;\text{\bf\normalsize ?}]
  \arrow|r|/@<-2pt>/[c`a;\text{\quad \ cpt.\@ open subgps.}]
\efig \]

The two remaining question marks correspond precisely to George Willis's question~\cite[Problem~4.3]{Willis07}.
It turns out that already for this class of groups, the answer to this question is negative.
More precisely, we will show the existence of two non-isomorphic groups $U(F)^+ \not\cong U(F')^+$ with isomorphic
edge stabilisers $U(F)^+_e \cong U(F')^+_{e}$.

\begin{prop}
    Assume that $F$ and $F'$ are two non-isomorphic subgroups of $\Sym(d)$ with equal point stabilisers $F_a = F'_a$
    acting on $d-1$ elements.
    Then $U(F)^+_e \cong U(F')^+_e$ but $U(F)^+ \not\cong U(F')^+$,
    \emph{i.e.}\@, the groups $U(F)^+$ and $U(F')^+$ are locally isomorphic but not isomorphic.
\end{prop}
\begin{proof}
    Since $F_a$ and $F'_a$ are isomorphic permutation groups, Proposition~\ref{pr:wr} implies that $U(F)^+_e \cong U(F')^+_e$.
    On the other hand, $F \not\cong F'$, and it follows from Theorem~\ref{th:global}(i) that $U(F)^+ \not\cong U(F')^+$.
\end{proof}

We now point out that we can indeed find two non-isomorphic subgroups $F, F' \leq \Sym(d)$ satisfying the conditions of
Section~\ref{sec:BM} (\emph{i.e.}\@ transitive and generated by their point stabilisers) that have equal point stabilisers.
The following example is the smallest possible (in terms of the permutation degree $d$).

\begin{example}
    Let $d=8$, let $F = \PSL(2,7) \leq \Sym(8)$ with its natural action on the projective line $\GF(7) \cup \{ \infty \}$,
    and let $F' = \AGammaL(1,8) \cong \AGL(1,8) \rtimes \langle \sigma \rangle \leq \Sym(8)$
    with its natural action on the affine line $\GF(8)$,
    where $\sigma$ is the generator of $\Gal\bigl(\GF(8)/\GF(2)\bigr) \cong C_3$.
    Then $F$ and $F'$ are two non-isomorphic doubly transitive permutation groups, and they have isomorphic point stabilizers
    $F_\infty \cong F'_0 \cong C_7 \rtimes C_3$ with an equivalent permutation action.
    Indeed, the equivalence of the permutation representation is induced by the bijection
    \[ \beta \colon \GF(7) \to \GF(8)^\times \colon n \mapsto \zeta^n \,, \]
    where $\zeta$ is a generator of the multiplicative group $\GF(8)^\times$.
    We conclude that the groups $U(F)^+$ and $U(F')^+$ are two non-isomorphic groups
    with isomorphic edge stabilizers and commensurable vertex stabilizers.
\end{example}

%% \begin{example}
%%  {\sffamily\bfseries This example is wrong!! \!:-(\ }
%%  Let $F = \Sym(3)$ with its natural action on $3$ elements, and let
%%  $F' = C_2 \wr \Sym(3)$ with its imprimitive wreath product action on $6$ elements.
%%  Then in fact, $F' \cong C_2 \times \Sym(4)$; it is not hard to check that
%%  the point stabilisers of $F$ and $F'$ are given by
%%  $F_a \cong C_2$ and $F'_a \cong C_2 \wr C_2 \cong D_8$.
%%
%%  By \cite[Section 3.2]{Burger-Mozes1}, we know that the maximal compact subgroup $U(F)^+_v$
%%  is obtained as the projective limit $\varprojlim A_n$ of the groups
%%  \[ \begin{cases}
%%      A_0 = F \,; \\
%%      A_n = F_a \wr A_{n-1} & \text{for } n \geq 1 \,,
%%  \end{cases} \]
%%  and similarly we write $U(F')^+_v = \varprojlim A'_n$.
%%  Now note that the wreath product is associative with respect to the imprimitive wreath
%%  product action.
%%  It follows from a straightforward induction that $A'_n = A_{2n+1}$ for all $n \geq 0$,
%%  and we conclude that $U(F)^+_v \cong U(F')^+_v$.
%%  On the other hand, $F \not\cong F'$, and it follows from Theorem~\ref{th:global}(i) that $U(F)^+ \not\cong U(F')^+$.
%% \end{example}

%%%%%%%%%%%%%%%%%%%%%%%%%%%%%%%%%%%%%%%%%%%%%%%%%%%%%%%%%%%%%%%%%%%%%%%%%%%%%%
\section{Rigidity for simple t.d.l.c.\@ groups}
%%%%%%%%%%%%%%%%%%%%%%%%%%%%%%%%%%%%%%%%%%%%%%%%%%%%%%%%%%%%%%%%%%%%%%%%%%%%%%

%This section was inspired by a reading of \cite{BarneaErshovWeigel}.

\subsection{Germs of automorphisms}\label{sec:Germ}

We recall some elements of terminology which were already defined in the introduction. 
Let $\mS$ denote the class of non-discrete topologically
simple totally disconnected locally compact groups. We say that two
groups $G_1, G_2 \in \mS$ are {\bf locally isomorphic} if
they contain isomorphic compact open subgroups. Moreover, a group $G
\in \mS$ will be called {\bf Lie-reminiscent} if any $H \in
\mS$ locally isomorphic to $G$ is in fact (globally)
isomorphic to $G$.

Any two compact open subgroups of a totally disconnected locally compact group $G$ are
commensurable. In particular, the commensurability class of any
compact open subgroup of $G$ depends only on its local structure,
in the sense that it can be reconstructed from any identity
neighbourhood. This commensurability class determines an algebraic object, namely the group $\acomm(G)$ of \textbf{germs of automorphisms}. This is defined as the quotient of the set of all isomorphisms $f: U \to V$ between compact open subgroups of $G$, divided by the equivalence relation which identifies two isomorphisms  $f_1: U_1 \to V_1$ and $f_2: U_2 \to V_2$ if they coincide on some open subgroup of $U_1 \cap U_2$. 

Notice that if $G$ is compact, hence profinite, then $\acomm(G)$ is nothing but the group of {\bf abstract commensurators} of $G$, which we denote by $\mathrm{Comm}(G)$. In fact, for any totally disconnected locally compact group $G$, we have $\acomm(G) =\acomm(U)$ for any compact open subgroup $U \leq G$. In particular $\acomm(G)$
depends only on the commensurability class of compact open subgroups of $G$. 

Since $G$ commensurates its compact open subgroups,
there is a canonical homomorphism
\[ \comm \colon G \to \acomm(G) \colon g \mapsto \comm(g).\]
Following \cite{BarneaErshovWeigel}, we endow $\acomm(G)$ with the \textbf{strong topology}, which is defined as the finest group topology which makes the homomorphism $\comm$ continuous. In this way $\acomm(G)$ is a topological group, which need not be Hausdorff in general. 

\begin{prop}\label{prop:Germ:basic} 
Let $G$ be a totally disconnected locally compact group. 
\begin{enumerate}[\rm (i)]
\item The kernel of $\comm : G \to \acomm(G)$ is the quasi-centre $\QZ(G)$ of $G$. 

\item If $\QZ(G)$ is closed, then $\acomm(G)$ is totally disconnected and locally compact. 

\item If $\QZ(G)$ is discrete, then $G$ is locally isomorphic to $\acomm(G)$.
\end{enumerate}
\end{prop}

\begin{proof}
\begin{compactitem}
\item[(i)]  Immediate from the definition of $\QZ(G)$, see  Definition~\ref{def:QZ}.

\item[(ii)] By definition of the strong topology, the group $\comm(G)$ is open in $\acomm(G)$. By (i) it is isomorphic to $G/\QZ(G)$, which is totally disconnected and locally compact provided $\QZ(G)$ is closed. In that case,  the profinite identity neighbourhoods of $G/\QZ(G)$ are also identity neighbourhoods of $\acomm(G)$. The desired result follows. 

\item[(iii)] If $\QZ(G)$ is discrete, then there is some compact open subgroup $U  \leq  G$ such that $U \cap \QZ(G) = 1$. Thus $G$ and $G/\QZ(G)$ are locally isomorphic and the result follows from (i). \qedhere
\end{compactitem}
\end{proof}

We emphasize that the quasi-centre need not be closed in general, even for simple groups. In fact, the examples of non-compactly generated simple groups
constructed by G.~Willis in \cite[\S3]{Willis07} have a dense
quasi-centre. However, the situation is more favorable in the case of compactly generated topologically simple groups, as illustrated by the following result due to Barnea--Ershov--Weigel (see also \cite[Prop.~4.3]{Caprace-Monod-monolith}). 

\begin{prop}[{\cite[Theorem~4.8]{BarneaErshovWeigel}}]\label{prop:QZ}
Any compactly generated totally disconnected locally compact group with dense quasi-centre admits a basis of identity neighbourhoods consisting of compact open {normal} subgroups. In particular,  if $G \in \mS$ and if $G$ is compactly generated, then $\QZ(G) = 1$.
\end{prop}
 
Let $\Aut(G)$ denote the group of homeomorphic automorphisms of $G$.
Clearly such automorphisms also commensurate compact open subgroups
of $G$, so that we get a canonical homomorphism
\[ \kappa \colon \Aut(G) \to \acomm(G) \,; \]
note that $\comm$ factors through $\kappa$ via the map $G \to \Inn(G) \leq \Aut(G)$.
Following the terminology introduced in
\cite{BarneaErshovWeigel}, we say that a simple group $G \in
\mS$ is {\bf rigid} if $\kappa$ is bijective. Equivalently, this means that every germ of automorphism of $G$ extends to a unique global automorphism. 

It is useful to have several equivalent ways to express rigidity at our disposal.
\begin{prop}[\cite{BarneaErshovWeigel}]\label{pr:rigid}
Let $G \in \mS$ such that $\QZ(G)=1$.
    Then the following conditions are equivalent:
    \begin{enumerate}[\rm (a)]
        \item
        $G$ is rigid, \emph{i.e.}\@ $\kappa$ is bijective.
        \item
        $\kappa$ is surjective.
        \item
        $\comm(G)$ is a normal subgroup of $\acomm(G)$.
        \item
        Any isomorphism between a pair of compact open subgroups of $G$ can be extended to a unique automorphism of $G$.
    \end{enumerate}
\end{prop}
\begin{proof}
%By Proposition~\ref{prop:QZ}, we have $\QZ(G) = 1$. 
   % In particular, we can invoke \cite[Proposition~3.7]{BarneaErshovWeigel}.
    Condition (d) is precisely the definition of rigidity as in \cite{BarneaErshovWeigel}, and hence
    \cite[Proposition~3.7]{BarneaErshovWeigel} shows that (a), (c) and (d) are equivalent.
    Of course (a) implies (b).
    Assume finally that (b) holds; then $\comm(G) = \kappa(\Inn(G)) \unlhd \kappa(\Aut(G)) = \acomm(G)$,
    so (c) holds.
\end{proof}

Our next goal is to clarify the relation between the notion of rigidity and of Lie-reminiscence for the elements of $\mS$.  We shall need the following two  lemmas.

\begin{lemma}\label{lem:QZopen}
Let $G \in \mS$ be compactly generated. Then every open subgroup of $\acomm(G)$ has trivial quasi-centre. In particular $\acomm(\acomm(G)) = \acomm(G)$.
\end{lemma}

\begin{proof}
Since $G \in \mS$ is compactly generated, we have $\QZ(G) = 1$ and hence $\QZ(U)=\Comm(U)=1$ for any compact open subgroup $U  \leq  G$. By \cite[Prop.~3.2(c)]{BarneaErshovWeigel}, it follows that $\QZ(\acomm(G)) = \QZ(\Comm(U)) = 1$. The first assertion follows since $\QZ(\acomm(G))$ contains the quasi-centre of any open subgroup of  $\acomm(G)$.
Since $\acomm(G)$ depends only on the commensurability class of the compact open subgroups of $G$, the second statement now follows from Proposition~\ref{prop:Germ:basic}(iii).
\end{proof}

\begin{lemma}\label{le:GH}
	Let $G,H \in \mS$ be two locally isomorphic compactly generated groups.
	Then $\langle \comm(G), \comm(H) \rangle \leq \acomm(G) = \acomm(H)$ is
	topologically simple.
\end{lemma}
\begin{proof}
	Since $G$ and $H$ are locally isomorphic, we have $\acomm(G) = \acomm(H) = \acomm$;
	recall that $\comm(G)$ and $\comm(H)$ are open subgroups of $\acomm$.
	Now let $N \neq 1$ be a closed normal subgroup of $\langle \comm(G), \comm(H) \rangle$.
	If $N \cap \comm(G)$ were trivial, then $N$ would be a discrete subgroup of $\langle \comm(G), \comm(H) \rangle$, which is impossible by Lemma~\ref{lem:QZopen}.
	Since $\comm(G) \cong G$ is topologically simple, this implies $\comm(G) \leq N$, and similarly
	$\comm(H) \leq N$.
	We conclude that $N = \langle \comm(G), \comm(H) \rangle$ as claimed.
\end{proof}
\begin{cor}
	Let $G, H \in \mS$ be compactly generated. If $G$ and $H$ are
	locally isomorphic, then there is a compactly generated group $S \in
	\mS$ in which both $G$ and $H$ embed as open subgroups.
\end{cor}
\begin{proof}
	This immediately follows from Lemma~\ref{le:GH} since $\comm$ is injective because $\QZ(G)$ and $\QZ(H)$ are trivial, see Proposition~\ref{prop:QZ}.
\end{proof}

In view of Proposition~\ref{prop:QZ}, the following result applies notably to any compactly generated group in $\mS$.

\begin{prop}\label{prop:AdmissibleRigid}
Let $G \in \mS$ be such that $\QZ(G)= 1$. Then:
\begin{enumerate}[\rm (i)]
\item Up to isomorphism, there is a unique rigid simple group $\tilde G \in \mS$ which is locally isomorphic to $G$; it can be defined as the intersection of all non-trivial closed normal subgroups of $\acomm(G)$. 

\item If $G$ is Lie-reminiscent, then it is rigid.

\item If $G$ is rigid, then any compactly generated group $H \in \mS$ locally
isomorphic to $G$ is isomorphic to some open subgroup of $G$.
\end{enumerate}
\end{prop}

It is important to keep in mind that the canonical rigid simple group $\tilde G$ need not be compactly generated, even if $G$ is so. This is notably illustrated by Theorem~\ref{thmi:L(G)} from the introduction. 

\begin{proof}[Proof of Proposition~\ref{prop:AdmissibleRigid}] 
\begin{compactenum}[(i)]
\item 
Let $\tilde G \leq \acomm(G)$ denote the intersection of all non-trivial closed normal subgroups of $\acomm(G)$.
Notice that any such normal subgroup is non-discrete because $\QZ(\acomm(G)) = 1$ by Lemma~\ref{lem:QZopen},
and hence meets the open subgroup $\comm(G)$ non-trivially.
Since $G \in \mS$, it follows that every non-trivial closed normal subgroup of $\acomm(G)$ contains $\comm(G)$ and is open;
in particular $\tilde G$ contains $\comm(G)$.

We next show that $\tilde G$ belongs to $\mS$ or, in other words, that it is topologically simple. So let $N$ be a closed normal subgroup of $\tilde G$. As before, Lemma~\ref{lem:QZopen} ensures that $N$ is non-discrete. Since $\tilde G$ contains the topologically simple group $\comm(G)$ as an open subgroup, we deduce that  the intersection $N \cap \comm(G)$ is non-trivial. But $\comm(G)$ being topologically simple, we infer that $\comm(G) \subseteq N$. In other words, every closed normal subgroup of $\tilde G$ contains $\comm(G)$;
hence the same is true for the intersection $M$ of all these closed normal subgroups. Thus $M$ is an open characteristic subgroup of $\tilde G$. In particular, it is a closed normal subgroup of $\acomm(G)$. It then follows from the definition that $M = \tilde G$,
whence $\tilde G$ is topologically simple as desired.

The fact that $\tilde G$ is rigid follows from Proposition~\ref{pr:rigid} since $\tilde G$ is normal in $\acomm(G) = \acomm(\tilde G)$ by construction. 

Let now $O \in \mS$ be any rigid group locally isomorphic to $G$. Since $O$ is topologically simple, its quasi-centre is either trivial or dense. Since $\QZ(G) =1$, every compact open subgroup of $G$ has trivial quasi-centre. The quasi-centre of $O$ can therefore not be dense. In particular $O$ embeds in $\acomm(O) \cong \acomm(G)$, see Proposition~\ref{prop:Germ:basic}. Let us identify $O$ with its image in $\acomm(G)$. Proposition~\ref{pr:rigid} guarantees that $O$ is open and normal; it therefore contains $\tilde G$. But $\tilde G$ is normal in $\acomm(G)$, and hence also in $O$. Since $O$ is simple, we infer that $O = \tilde G$, as desired.

    \item
	Let $\tilde G \in \mS$ be the rigid group provided by (i). Since $G$ and $\tilde G$ are locally isomorphic, we infer that if $G$ is  Lie-reminiscent, then it must be isomorphic to $\tilde G$. Hence $G$ is rigid, as desired.

\item
	Assume now that $G$ is rigid. Then $\comm(G)$ is normal in $\acomm(G)$ and any compactly generated group $H \in \mS$ locally
	isomorphic to $G$ embeds as an open subgroup in $\acomm(G)$
	via the injection $\comm \colon H \to \acomm(H) = \acomm(G)$.
	Therefore the open subgroups $\comm(H)$ and $\comm(G)$ of $\acomm(G)$ meet
	non-trivially, and since $\comm(G)$ is normal while $H$ is
	topologically simple, we obtain $\comm(H) \subseteq \comm(G)$ as desired.
    \qedhere
\end{compactenum}
\end{proof}

\begin{remark}
It is interesting to point out that the above proof implies moreover that, if $G$ is \emph{abstractly} simple, then so is $\tilde G$. 
\end{remark}

\begin{remark}
The group $\tilde G$ is called the \textbf{open normal core} of $\acomm(G)$ in \cite{BarneaErshovWeigel}, and it is denoted by $\mathrm{Onc}(\Comm(U)_S)$, where $U$ is a compact open subgroup of $G$. It is defined in \emph{loc.~cit.} as the intersection of all open normal subgroups of $\acomm(G) = \Comm(U)$; the above proposition shows that these two definitions coincide in our setting.
\end{remark}

\subsection{A local property of compactly generated rigid simple groups}

As mentioned above, the rigid group $\tilde G$ appearing in Proposition~\ref{prop:AdmissibleRigid}(i) need not be compactly generated  in general, even if $G$ is so. We will study this question in detail for Burger--Mozes universal groups in order to prove Theorem~\ref{thmi:L(G)} from the introduction. In that the study, the following general fact will be helpful; it is implicitly contained in Section~8 from  \cite{BarneaErshovWeigel}. 

\begin{prop}\label{prop:Aut}
Let $G \in \mS$ be compactly generated and rigid. Then for any compact open subgroup $U  \leq  G$ and any characteristic open subgroup $V  \leq  U$, the automorphism group $\Aut(U)$ embeds as a  subgroup of $\Aut(V)$ whose index is at most countable.
\end{prop}

\begin{proof}
Since $V$ is characteristic in $U$, there is a map $\Aut(U) \to \Aut(V)$. Since $G$ has trivial quasi-centre (see Proposition~\ref{prop:QZ}), so does $U$ and hence the above map is injective. All we need to show is that its image has countable index. 

We shall identify the groups $G$, $\Aut(U)$ and $\Aut(V)$ with their canonical images in $\mathscr L(G)$.  Since $G$ is compactly generated, every open subgroup has countable index. In particular the normaliser of $U$ has countable index in $G$, and hence $U$ has countably many conjugates in $G$. 

Since $G$ is normal in $\mathscr L(G)$ (see Proposition~\ref{pr:rigid}), every conjugate of $U$ in $\mathscr L(G)$ is contained in $G$. Lemma~\ref{lem:QZopen} guarantees that the normaliser of $U$ in $\mathscr L(G)$ is nothing but $\Aut(U)$. Therefore,  we infer that  $\Aut(U)$ has countable index in $\mathscr L(G)$. The latter index being clearly an upper bound for the index of $\Aut(U)$ in $\Aut(V)$, the desired result follows.
\end{proof}

\subsection{Non-rigidity of tree-automorphism groups satisfying Tits' independence property}

\begin{theorem}\label{thm:NonRigid}
Let $T$ be a locally finite tree and $G \leq \Aut(T)$ be a non-trivial compactly generated simple closed subgroup satisfying Tits' independence property. Then $G$ is not rigid, and in particular it is not Lie-reminis\-cent.
\end{theorem}

\begin{proof}
Upon replacing $T$ by a minimal $G$-invariant subtree, there is no loss of generality in assuming the $G$-action to be minimal. In particular it is cocompact by Lemma~\ref{le:ccpt}.

We need to consider a certainly family of sub-trees of $T$, associated to each pair $(v, A)$ consisting of a vertex $v \in V(T)$ and a set $A \subseteq E(v)$ of edges containing $v$. To such a pair $(v, A)$, we associate the subtree 
$$h(v, A) \subseteq T$$
whose vertex set is 
$$\{v\} \cup \{w \in V(T) \; | \; [v, w] \cap A \neq \varnothing\}.$$
In other words, $h(v, A)$ is the subtree containing $v$ whose vertices different from $v$ are separated from $v$ by an edge in $A$.
In particular, $h(v, E(v)) = T$ and $h(v, \varnothing) = \{v\}$. Similarly, if $A =\{e\}$ consists of a single edge, then $h(v, \{e\})$ is the union of $e$ with  the half-tree determined by $e$ and not containing $v$. 

Since $G$ acts  cocompactly on $T$, it has finitely many orbits of vertices and edges and, hence, finitely many orbits of subtrees of the form $h(v, A)$ as above.
We denote these orbits by $\Omega^1, \dots, \Omega^k$.
Given $h \in \Omega^i$, we denote by $G_{(h)}$ the pointwise stabiliser of the subtree $h$. 

Fix a base vertex $v \in T$.
Given $n>0$, we denote by $U_n \leq G$ the pointwise stabiliser of the ball of radius $n$ around $v$.
Let $T_n$ be the fixtree of $U_n$, and let $T_n^{[0]}$ be the \emph{thick part} of $T_n$, i.e.\@ the subset  consisting of those vertices of $T_n$ all of whose neighbours in $T$ also belong to $T_n$. In particular, the $(n-1)$\nobreakdash-ball around $v$ is contained in $T_n^{[0]}$. Notice that $T_n^{[0]}$ need not be connected \emph{a priori}. Moreover $T_n \setminus T_n^{[0]}$ is necessarily non-empty, since otherwise we would have $T_n = T_n^{[0]} = T$ and $G$ would be discrete, hence trivial. 

In fact, we claim more precisely that $T_n$ coincides with the convex hull of $T_n \setminus T_n^{[0]} $. Indeed, consider an edge $e \in E(T_n)$, and let $h^+$ and $h^-$ be the two half-trees of $T$ determined by $e$. Then we claim that $h^+$ is not entirely contained in $T_n$.
Indeed, suppose that $h^+$ were contained in $T_n$; then $U_n$ would fix $h^+$ pointwise.
But then the pointwise stabilizer $G_{(h^+)}$ would be open; this would force the pointwise stabilizer $G_{(h^-)}$ to be finite, since $G_{(e)} \cong G_{(h^+)} \times G_{(h^-)}$ by Tits' independence property. Since $G$ has trivial quasi-centre by Proposition~\ref{prop:QZ}, we deduce that $G_{(h^-)}$ is trivial.
It is easy to see that this would imply that $G$ itself is trivial, which is absurd. Thus $T_n$ does not contain $h^+$. This is equivalent to saying that $h^+$ contains some vertex $v^+ \in T_n \setminus T_n^{[0]}$. Similarly, one shows that $h^-$ contains some vertex $v^- \in T_n \setminus T_n^{[0]}$. Thus $e \subset [v^+, v^-]$. Since $e$ was an arbitrary edge of $T_n$, this confirms the claim that  $T_n$ is the convex hull of $T_n \setminus T_n^{[0]}$.

\medskip 
For each vertex $v \in T_n \setminus T_n^{[0]}$, the intersection $A = E(v) \cap E(T_n)$ is a proper subset of $E(v)$. In particular $h\big(v, E(v) \cap E(T_n)\big)$ is a proper subtree of $T$ containing $T_n$. We define  
$$H_n = \bigg\{ h\big(v, E(v) \cap E(T_n)\big) \; | \; v \in T_n \setminus T_n^{[0]}\bigg\}.$$
as the collection of all subtrees of that form. Moreover, for each $i$ and each $n$, we set 
$\Omega_n^i = \Omega^i \cap H_n$. Notice that $\Omega_n^i$ can be finite or infinite. 

Now Tits' independence property implies that 
\[ U_n \cong \prod_{h \in H_n} G_{(h)} = \prod_{i=1}^k \Bigl(  \prod_{h \in \Omega_n^i} G_{(h)} \Bigr) . \]
Set $N_n = \norma_G(U_n)$; then $U_n \leq G_v \leq N_n$.
Moreover, let
%$A_n \leq \Aut(G)$ be the subgroup consisting of all automorphisms of $G$ which normalise $U_n$ and globally preserve the above product decomposition. 
\begin{multline*}
	A_n = \{ \varphi \in \Aut(G) \mid \varphi(U_n) = U_n, \text{ and for all } i \in \{ 1,\dots,k \} \\[-.2ex]
	\text{ and all } h \in \Omega_n^i, \text{ we have } \varphi(G_{(h)}) = G_{(h')} \text{ for some } h' \in \Omega_n^i \} .
\end{multline*}

Since $G$ is simple it injects in $\Aut(G)$ and it will be convenient to abuse notation and identify $G$ with its image in $\Aut(G)$.
Modulo this convention, we notice that $N_n$ is contained in $A_n$: Indeed, the group $N_n$ acts on the fixtree $T_n$. Therefore, it preserves $T_n^{[0]}$ and, hence, it permutes the elements of $\Omega_n^i$ for all $i \in \{1, \dots, k\}$. Thus $N_n \leq A_n$. Since on the other hand, the definition of $A_n$ implies that $A_n \cap G$ normalizes $U_n$, we conclude that $N_n = A_n \cap G$, which implies in particular that $N_n$ is a normal subgroup of $A_n$. 

By definition  we have a canonical homomorphism
\[ \pi_n^i  : A_n \to   \Sym(\Omega_n^i) \]
for each $i = 1, \dots, k$. The product of these defines a homomorphism
\[ \pi_n  : A_n \to \prod_{i=1}^k  \Sym(\Omega_n^i). \]
Observe that $\ker(\pi_n) \cap N_n $ acts trivially on the set $T_n \setminus T_n^{[0]}$. Since $T_n$ is the convex hull of $T_n \setminus T_n^{[0]}$ by the above, we deduce that $\ker(\pi_n) \cap N_n$ acts trivially on $T_n$, and is thus contained in $U_n$. This shows that  $\ker(\pi_n) \cap N_n = U_n $ for all $n$. Moreover, since $U_n \neq G_v$ for all sufficiently large $n$,
this implies that $\pi_n(N_n)$ is non-trivial for all sufficiently large $n$.
%Also observe that if $\pi_n^i(G_v)$ is non-trivial for some $i$ and some~$n$, then $\pi_m^i(G_v)$ is non-trivial for all $m \geq n$.

\smallskip

Suppose now for a contradiction that $G$ is rigid. Then every abstract commensurator of $U_n$ extends to an automorphism of $G$. In particular any automorphism of $U_n$  permuting isomorphic factors extends to an element of $A_n$. This implies that the above map $\pi_n^i$ is surjective for all $i$ and $n$. 

Let $i$ and $n$ be such that $\Omega_n^i$ is infinite (and hence countably infinite), and assume that $\pi_n^i(N_n) \neq 1$. Recall that $G$ is second countable since it is metrisable and compactly generated. In particular $N_n$ is second countable, and the discrete image $\pi^i_n(N_n)$ is therefore at most countable. In particular it is a countable normal subgroup of the uncountable group $\Sym(\Omega_n^i)$.
By the Baer--Schreier--Ulam theorem \cite{Baer,SchreierUlam},
$\Sym(\Omega_n^i)$ has only two proper non-trivial normal subgroups,
namely the subgroup of all finitary permutations, and the subgroup of alternating finitary permutations;
both normal subgroups are locally finite and infinite. We deduce that $\pi_n^i(N_n)$ is locally finite and, hence, that $N_n$ is \textbf{locally elliptic},
\emph{i.e.}\@ every compact subset is contained in a compact subgroup. Since $N_n $ is open in $G$, Lemma~\ref{lem:basic}(iv) guarantees that this can only be true if $N_n$ is compact;
but then $\pi_n^i(N_n)$ is finite, so $\pi_n^i(N_n)$ is trivial after all.

We infer that $\pi_n^i(N_n) = \{1\}$ for all $i$ and $n$ such that $\Omega_n^i$ is infinite;
in particular $\pi_n(N_n) $ is finite for all $n$. 
%
% Similar arguments show that $\pi_n^i(N_n)$ is  transitive on $\Omega_n^i$ provided $\pi_n^i(N_n)$ is non-trivial. In fact, the group $\pi_n^i(N_n)$ is either  $\Sym(\Omega_n^i)$ or $\Alt(\Omega_n^i)$ if  $\pi_n^i(N_n)$ is non-trivial and $|\Omega_n^i | \geq 5$.
Similarly, if the group $\pi_n^i(N_n)$ is non-trivial and $|\Omega_n^i | \geq 5$, then it coincides with either $\Sym(\Omega_n^i)$ or $\Alt(\Omega_n^i)$.

\medskip
Our next goal is to show that $\pi_n^i(G_v)$ is either $\Sym(\Omega_n^i)$ or $\Alt(\Omega_n^i)$ for some fixed $i$ and infinitely many values of $n$. To this end, we first notice that $N_n$ is compact, since $\ker(\pi_n) \cap N_n = U_n$ is compact and since $\pi_n(N_n) $ is finite. Thus for each $n$ there is some vertex $v_n$ such that $N_n \leq G_{v_n}$, and hence $G_v \leq N_n \leq G_{v_n}$. Since $G$ is unimodular (because it is simple) and acts cocompactly on $T$, we deduce that 
the values $[G_{v_n} : G_{v}]$ are bounded, and hence
\[ s = \sup_n [N_n : G_v] < \infty. \]
Recall that, for $m \geq 5$, the only non-trivial subgroup of $\Sym(m)$ of index $< m$  is $\Alt(m)$, and that $\Alt(m)$ has no non-trivial subgroup of index~$<m$.  It follows that for sufficiently large $n$, if $|\Omega_n^i| > \max\{5, s\}$ and $\pi_n^i(N_n)$ is non-trivial, then
$\pi_n^i(G_v)$ also coincides with either $\Sym(\Omega_n^i)$ or $\Alt(\Omega_n^i)$.

On the other hand, we have 
% \[ |\pi_n(N_n)| \geq |\pi_n(G_v)| = |G_v / U_n|, \]
\[ |\pi_n(G_v)| = |G_v / U_n| \leq  |\pi_n(N_n)| , \]
which shows that $ |\pi_n(N_n)|$ tends to infinity with $n$. In particular, the subset $I \subset \{1, \dots, k\}$ consisting of those $i$ such that $| \pi^i_n(N_n) | > \min\{60, s!\}$ for infinitely many values of $n$, is non-empty. In view of the conclusion of the preceding paragraph, we deduce that  if $i \in I$, then $\pi^i_n(G_v) = \Sym(\Omega_n^i)$ or $\Alt(\Omega_n^i)$ for infinitely many values of $n$, as desired. 

\medskip
The final contradiction is now obtained as follows. Since $G_v$ preserves each sphere around $v$, it follows that for each $i$ and $n$ with $\pi^i_n(G_v) = \Sym(\Omega_n^i)$ or $\Alt(\Omega_n^i)$, there is some $m \geq n$ such that the permutation representation $\pi_n^i : G_v \to \Sym(\Omega_n^i)$ is a sub-representation of the permutation action of $G_v$ on the $m$-sphere around $v$. The latter action is imprimitive, with minimal blocks of imprimitivity of size at most $d-1$, where $d$ is the maximal valence of a vertex of $T$;
in particular, the number of blocks of the sphere of radius $m$ tends to infinity with $m$.
On the other hand, the number of sets $\Omega_n^i$ is bounded by $k$, for each $n$.
Therefore, the pigeonhole principle implies that there is some $i \in I$ for which $\pi_n^i(G_v)$ is an imprimitive subgroup of $\Sym(\Omega_n^i)$ for all $n$, whose order tends to infinity with $n$.
This finally contradicts the fact that for this $i$ and some $n$ large enough, we have $\pi^i_n(G_v) =  \Sym(\Omega_n^i)$ or $\Alt(\Omega_n^i)$.
\end{proof}

%%%%%%%%%%%%%%%%%%%%%%%%%%%%%%%%%%%%%%%%%%%%%%%%%%%%%%%%%%%%%%%%%%%%%%%%%%%
\section{Commensurators of self-replicating wreath branch groups}\label{sec:Comm}
%%%%%%%%%%%%%%%%%%%%%%%%%%%%%%%%%%%%%%%%%%%%%%%%%%%%%%%%%%%%%%%%%%%%%%%%%%%

\subsection{The vocabulary of branch groups and rooted trees}

We recall basic notions from the theory of branch groups.

\begin{defn}
	A tree $\mathcal T$ is called a {\bf regular rooted tree} of degree $d$ (with \textbf{root} $r$), if $r$ is a vertex
	of valency $d$, and every other vertex of the tree $\mathcal{T}$ has valency $d+1$.
	In particular, every automorphism of $\mathcal{T}$ fixes $r$, and the group $\Aut(\mathcal T)$ is a compact, and hence profinite, group. 
	A vertex $v \in V(\mathcal{T})$ is called at {\bf level $n$} if it has (graph-theoretical) distance $n$ from the root $r$. 
Given a vertex $v$ of level $n$, we denote by $\mathcal T_v$ the subtree of $\mathcal T$ consisting of those vertices of level~$\geq n$ which are separated from the root by~$v$. 

Given a group $W$ acting on a rooted tree $\mathcal T$,
	we denote by $\st_W(n)$ the pointwise stabiliser of the sphere of radius $n$ around the root.
	Given a vertex $v$ of $\mathcal T$ of level $n$, the \textbf{restricted stabiliser} of $v$ in $W$,
	denoted by $\rist_W(v)$,
	is defined as the subgroup of $\st_W(n)$ acting trivially on $\mathcal T_w$ for every vertex $w$ 	of level $n$ different from $v$. 
	We say that $W$ is {\bf level-transitive} if it acts transitively on the set of vertices at level $n$, for each $n$.

(All these definitions can be extended to more general rooted trees, and are particularly useful for spherically
	homogeneous trees, but we will mostly need the rooted regular trees in what follows.)
\end{defn}
\begin{defn}
	A profinite group $W$ is called a {\bf branch group}, if there is a rooted tree $\mathcal{T}$ 
	and an embedding of $W$ into $\Aut(\mathcal{T})$ as a closed subgroup, such that $W$ is level-transitive on $\mathcal{T}$, and such that for each $n \geq 1$ the subgroup generated by the restricted stabilisers of all vertices at level~$n$ is of finite index (and hence open) in $W$.  We say that $W$ is \textbf{saturated}, if $\st_{W}(n)$ is a characteristic subgroup of $W$ for all $n \geq 0$.

\end{defn}
\begin{remark}
\begin{compactenum}[(i)]
\item If $W$ is a branch group, then it has trivial quasi-centre (this follows \emph{e.g.} from \cite[Theorem~2(c)]{Grigorchuk}).

	    \item
		Some authors use the terminology {\em rigid stabiliser} instead of restricted stabiliser,
		and they denote this by $\mathrm{rist}_W(v)$ instead of $\rist_W(v)$.
	    \item
		It is possible to define branch groups without explicitly referring to the associated rooted tree;
		see, for example, \cite[Section~5]{Grigorchuk}.
\end{compactenum}
\end{remark}

\subsection{Self-replicating wreath branch groups and their auto\-mor\-phisms}\label{sec:branch}

In the rest of the paper, we shall focus on a special class of profinite branch groups which we now introduce. 

Let $d >1$ and $D  \leq  \Sym(d)$ be a finite transitive permutation group. 
Set $D_1 = D$ and $D_{n+1} = D_1 \wr D_n$ for all $n \geq 1$, where the wreath product is taken with its imprimitive wreath product action
(which is the natural action of $D_n$ on the $d^n$ leaves of the corresponding finite rooted tree).
The family $(D_n)_{n >0}$ naturally constitutes a projective system of finite groups. We denote by $W(D) = \varprojlim D_n$ the corresponding profinite group and call it the \textbf{wreath branch group} determined by $D$; see also Proposition~\ref{pr:wr} above.

We shall view $W(D)$ as a closed subgroup of the automorphism group of the regular rooted tree $\mathcal T$ of degree $d$ on which $W(D)$ acts continuously and faithfully. Since $D$ is transitive, it follows that $W(D)$ acts on $\mathcal T$ as a branch group. Notice that for all $v \in V(\mathcal T)$, there is an isomorphism $\mathcal T_v \to \mathcal T$ which conjugates $\rist_{W(D)}(v) $ onto $ W(D)$; in particular $\rist_{W(D)}(v) \cong W(D)$. We shall refer to this property by saying that $W(D)$ is \textbf{self-replicating}.

(Although this is not relevant to our purposes, we point out that, up to isomorphism, the group $W(D)$ is the unique closed branch subgroup $W \leq \Aut(\mathcal T)$ such that $W / \st_W(1) \cong D$ and that $W$ is self-replicating.) 

\begin{lemma}\label{lem:JustInfinite}
$W(D)$ is just-infinite if and only if $D$ is perfect.
\end{lemma}

\begin{proof}
Immediate from  \cite[Theorem~4]{Grigorchuk}.
\end{proof}

Although we shall not need it in this paper, we mention the following striking characterisation, due to M.~Burger and Sh.~Mozes (see \cite{Moz}).

\begin{theorem}[Burger--Mozes]
The profinite group $W(D)$ is topologically finitely generated if and only if $D \leq \Sym(d)$ is perfect and fixed-point-free.
\end{theorem}

\medskip
Automorphisms of wreath branch groups have been studied in~\cite{Lavreniuk99}. We denote by $W^k$ the direct product of $k$ copies of $W$.

\begin{lemma}\label{lem:Aut(W)}
The group $W = W(D)$ is saturated and for any $k>0$, we have $\Aut(W^k) = \Aut(W) \wr \Sym(k)$. Moreover we have 
\[ \Aut(W) = \norma_{\Aut(\mathcal T)}(W) \]
and
\[ \Out(W) \cong \prod_{n>0} \norma_{\Sym(d)}(D)/D. \]
In fact $\Aut(W)$ is contained in $W(\norma_{\Sym(d)}(D))$ and coincides with the projective limit $\varprojlim A_n$, where the sequence of groups $A_n \leq \Sym(d^n)$ is defined inductively by 
$A_1 = \norma_{\Sym(d)}(D)$ and $A_{n+1} = B_{n} \rtimes A_n \leq A_1 \wr A_n$, and $B_{n}$ is the subgroup of $\big(\norma_{\Sym(d)}(D)\big)^{d^n}$ consisting of those $d^n$-tuples $(\alpha_1, \dots, \alpha_{d^n})$ such that $\alpha_i \equiv \alpha_j \mod D$ for all $i, j \in \{1, \dots, d^n\}$.
\end{lemma}

\begin{proof}
The main result from \cite{Lavreniuk99}  guarantees that $W(D)$ is saturated and that $\Aut(W) = \norma_{\Aut(\mathcal T)}(W)$.  This is also established in \cite[Theorem~8.2]{LN02}, where it is moreover proved that $\Aut(W^k) = \Aut(W) \wr \Sym(k)$.

Set $A =  \norma_{\Aut(\mathcal T)}(W) \leq \Aut(\mathcal T)$, $A_n = A/\st_A(n)$ and $W_n = W/\st_W(n)$ for all $n >0$. Clearly $A \cong \varprojlim A_n$.  We view $A_n$ and $W_n$ as subgroups of $\Aut(\mathcal T_n)$, where $\mathcal T_n$ denotes the finite rooted tree consisting of the truncation of $\mathcal T$ at level~$n$. 

We have $A_n = \norma_{\Aut(\mathcal T_n)}(W_n)$. In particular $A_1 \cong  \norma_{\Sym(d)}(D)$ and $A_{n+1}$ is contained in $A_1 \wr A_n = \big( \bigoplus_{i=1}^{d^n} A_1\big) \rtimes A_n$. Referring to the latter semi-direct decomposition, we shall describe the elements of $A_{n+1}$ as $(d^n+1)$-tuples of the form $\big( (\alpha_1, \dots, \alpha_{d^n}), w\big)$, where $\alpha_i \in A_1$ and $w \in A_n$.

Consider the elements $\alpha=\big( (\alpha_1, \dots, \alpha_{d^n}), 1\big)$ and $g = \big( (1, \dots, 1), w\big)$ of $A_{n+1}$, where $\alpha_i \in A_1$ and $w \in W_n$. Since $W_{n+1}$ is normal in $A_{n+1}$ and contains $g$, we deduce that the commutator $[\alpha, g]$ belongs to $W_{n+1}$. In particular we deduce that $\alpha_i\inv \alpha_{i\cdot w} \in D$ for all $i$. Since $W_n$ is a transitive subgroup of $\Sym(d^n)$, we deduce that $\alpha_i \equiv \alpha_j \mod D$ for all $i$ and $j$. This confirms the claim that $A_{n+1} \cong B_n \rtimes A_n$.

Consider the subgroup $C_n =  D^{d^n} \leq B_n$. Notice that $C_n$ is normal in $A_{n+1}$. Moreover the quotient $A_{n+1}/ C_n$ is isomorphic to the direct product $B_n/C_n \times A_n \cong \norma_{\Sym(d)}(D)/D \times A_n$. Since $C_n$ is contained in $W_{n+1}$, it follows inductively that $A_{n+1}/W_{n+1} \cong \prod_{i=1}^{n+1} \norma_{\Sym(d)}(D)/D$. Therefore, we have  $\Out(W) \cong \prod_{n>0} \norma_{\Sym(d)}(D)/D$ as desired.
\end{proof}

\subsection{Abstract commensurators and Higman--Thompson groups}
We shall need to consider the so called Higman--Thompson groups, an excellent introduction to which  can be consulted in \cite[\S4]{Brown}.  (Numerous other excellent references are available concerning this subject, see notably  \cite[\S3]{Rover} and references cited there). For each $d > 1$ and $k>0$, there is a Higman--Thompson group  $V_{d,k}$ whose definition will be recalled below. Remark that K.~Brown  \cite{Brown} denotes this group by $G_{d,k}$, but we prefer the notation $V_{d,k}$ in order to avoid confusion with (subgroups of) our group $G$. The groups $V_{d, k}$ were first introduced by G.~Higman in \cite{Higman} as generalisations of a group~$V$ earlier introduced by R.~Thompson.  It is shown in \cite[Theorem~5.4]{Higman} that the derived group of $V_{d, k}$ is simple and finitely presented, and has index one or two according as  $d$ is even or odd (see also \cite[Theorems~4.16 and~4.17]{Brown}). 

For the sake of completeness, we briefly describe how the groups $ V_{d, k}$ can be defined. To this end, we denote by $\mathcal T_d$ the regular rooted tree of degree $d$ and consider the graph $\mathcal T_{d, k}$ which is the unique rooted tree having $k$ vertices of level~$1$ and such that $(\mathcal T_{d, k})_v \cong \mathcal T_d$ for all vertices $v$ of level~$1$.  In a similar way as in Section~\ref{sec:BM}, we fix a colouring 
$$ i : E(\mathcal T_{d, k}) \to \{1, \dots, d\}$$
of the edge-set of $\mathcal T_{d, k}$ such that, for each vertex $v$ of level~$n>0$, the restriction of $i$ to the set $D(v)$ is bijective, where $D(v)$ is the set of those edges joining $v$ to one of its $d$ neighbours of level~$n+1$.  An \textbf{almost automorphism} of the tree $\mathcal T_{d, k}$ is a triple of the form $(A, B, \varphi)$, where $A$ and $B$ are finite subtrees of $\mathcal T_{d, k}$ containing the root such that $| \bd A | = | \bd B |$, and $\varphi$ is an isomorphism of forests $\mathcal T_{d, k} \setminus A \to \mathcal T_{d, k} \setminus B$. The \textbf{group of almost automorphisms} of $\mathcal T_{d, k}$, denoted by $\mathrm{AAut}(\mathcal T_{d, k})$, is the quotient of the set of all almost automorphisms by the relation which identifies two almost automorphisms $(A, B, \varphi)$ and $(A', B', \varphi')$ if there exists some finite subtree $A''$ containing $A \cup A'$ and such that $\varphi$ and $\varphi'$ coincide on $\mathcal T_{d, k} \setminus A''$.   It is easy to verify that  $\mathrm{AAut}(\mathcal T_{d, k})$ is indeed a group. 

An element of  $\mathrm{AAut}(\mathcal T_{d, k})$  is called \textbf{rooted} if it can be represented by an almost automorphism $(A, B, \varphi)$ such that for each vertex $v$ of  $\mathcal T_{d, k} \setminus A$, the composed map 
$$ i |_{D(\varphi(v))} \circ \varphi \circ \big( i |_{D(v)} \big)\inv $$
is the identity permutation of $\{1, \dots, d\}$. The \textbf{Higman--Thompson group} $V_{d, k}$ is defined as the subgroup of  
$\mathrm{AAut}(\mathcal T_{d, k})$ consisting of the rooted elements. Its isomorphism type is independent of the choice of the colouring~$i$.

Following K.~Brown, we shall also consider the torsion-free subgroup $F_{d,k} \leq  V_{d, k}$ consisting of the so-called \textbf{order-preserving} elements (see \cite[\S4]{Brown}). In order to define it,  we first remark that a rooted almost automorphism $(A, B, \varphi)$ is uniquely determined by the bijection $\bd A \to \bd B$ of the leaf sets of $A$ and $B$ that it defines. Now we fix a \textbf{planar embedding} of the graph $\mathcal T_{d, k} $ in the plane $\RR^2$, \emph{i.e.} an embedding where edges are represented by segments and two edges intersect in a point if and only if this point represents a common vertex.
Moreover, we choose this embedding in such a way that the root coincides with the origin and that all vertices of level $n$ lie on the line $\{(x, y) \in \RR^2 \mid y = n\}$. Once this planar embedding is fixed, the leaf set $\bd A $ of each finite subtree $A \subset \mathcal T_{d, k}$ inherits a well-defined ordering, say from left to right. A rooted almost automorphism $(A, B, \varphi)$ is then called \textbf{order-preserving} if the induced bijection $\bd A \to \bd B$ preserves that ordering. The subgroup $F_{d, k}  \leq  V_{d, k}$ consists by definition of those elements which can be represented by an order-preserving rooted almost automorphism. One verifies easily that these form indeed a subgroup. The group which is commonly known as \textbf{Thompson's group $F$} is nothing but $F_{2, 1}$, while Thompson's group $V$ is $V_{2, 1}$. It can be shown that  $F_{d, k}$ is isomorphic to $F_{d, 1}$ for any $k>0$, see \cite[Proposition~4.1]{Brown}, but we will not need this fact here.

Notice that every automorphism of $\mathcal T_{d, k} $ defines a unique element of the group $\AAut(\mathcal T_{d, k})$. Thus there is an embedding $\Aut(\mathcal T_{d, k} ) \to \AAut(\mathcal T_{d, k})$; its image intersects the subgroup $F_{d, k}$ trivially. 

Let us finish this discussion by a final comment concerning $V_{d, k}$. We have mentioned above that $V_{d, k}$ is finitely generated, that it is simple if $d$ is even and possesses a simple subgroup of index~$2$ when $d$ is odd. This quotient of order~$2$ can easily be understood from the above definition. Indeed, we have pointed out that a rooted almost automorphism  $(A, B, \varphi)$ is uniquely determined by a bijection $\partial A \to \partial B$. We can post-compose this bijection with the unique order-preserving bijection $\bd B \to \bd A$. In this way, we see that $(A, B, \varphi)$ is uniquely determined by a permutation of $\bd A$. This permutation is either even or odd. One sees that when $d$ is even, we can always enlarge $A$ and $B$ (by ``unfolding a leaf'' to the next level) so as to represent a rooted almost automorphism by an even permutation, while when $d$ is odd, the even or odd character of this permutation is independent of the chosen representative.

\bigskip
Now we return to the setting of the preceding paragraph and consider the wreath branch group  $W = W(D)$ associated to a finite permutation group $D \leq \Sym(d)$ as before. For $k \geq 1$, we denote by $W^k$ the direct product of $k$ copies of $W$. We assume throughout that $D$ is transitive. Since $W$ is self-replicating, we have  $\st_W(n) \cong W^{d^n}$ for all $n \geq 0$. In particular $\st_W(1)$ is isomorphic to the direct product of $d$ copies of $W$. Thus it follows from \cite[Theorem~1.1]{Rover} that the group of abstract commensurators $\Comm(W) = \acomm(W)$ contains a copy of the Higman--Thompson group $V_{d, 1}$. More generally, we have the following.

\begin{lemma}\label{lem:HT}
For each $k \geq 1$, the action of the Higman--Thompson group $V_{d, k}$ by almost automorphisms on $\mathcal T_{d, k}$ defines an embedding  $V_{d, k} \to \acomm(W^k)$.  
\end{lemma}

\begin{proof}
By definition, the group $W^k$ can be viewed as a closed subgroup of $\Aut(\mathcal T_{d, k})$ fixing pointwise the $k$ vertices of level~$1$. Moreover, for each vertex $v$ different from the root, we have $\rist_{W^k}(v) \cong W$. Therefore, every rooted automorphism of $T_{d, k}$ indeed defines a germ of automorphism of $W^k$, and it follows from the definitions that this induces an embedding  $V_{d, k} \to \acomm(W^k)$.
\end{proof}

%it follows from the definitions that the Higman--Thompson group $V_{d, k}$ embeds in $\acomm(W^k)$. 

We need to consider some specific open subgroups of $\acomm(W^k)$, which we shall now describe. For all $k >0$ and $n \geq 0$, we set
\[ A_{k, n+1}  = \Aut(W^{kd^n}) =  \Aut(W) \wr \Sym(k d^n). \]
The latter equality follows from Lemma~\ref{lem:Aut(W)}. Thus $A_{1, 1} = \Aut(W)$. This lemma also guarantees that $W$ is saturated and, more generally, that the wreath branch group $W \wr \Sym(k d^n)$ is saturated. This implies moreover that $\Aut(W)$ embeds in $\Aut(\st_W(1) ) = \Aut(W^d) = \Aut(W) \wr \Sym(d)$. More generally, we have 
\begin{align*}
A_{k, n} &= \Aut(W^{kd^{n-1}}) \\
&=   \Aut(W) \wr \Sym(k d^{n-1}) \\
&\leq  \big(\Aut(W) \wr \Sym(d) \big) \wr \Sym(k d^{n-1}) \\
&\leq   \Aut(W) \wr \Sym(k d^{n})\\
&=  A_{k, n+1} \,.
\end{align*}
Thus we have natural inclusions $A_{k,1} \leq A_{k,2} \leq \dotsm$ for all $k >0$, and the inductive limit 
\[ A_k = \bigcup_n A_{k,n} \]
is thus a group. Notice that $A_k$ embeds as a subgroup in $\acomm(W^k)$; we shall identify it with its image.
In particular $A_k$ inherits the strong topology from $\acomm(W^k)$.
Remark moreover that $A_k$ is open in $\acomm(W^k)$ since it contains $\comm(W^k)$.

In a similar way as above, we define $A_{k, n+1}^+ = \Aut(W) \wr \Alt(k d^n)$, which is a subgroup of index two in $A_{k, n+1}$. By restricting the inclusion $A_{k, n} \leq A_{k, n+1}$, we have $A_{k, n}^+ \leq A_{k, n+1}$. We also define an open subgroup $A_k^+ \leq A_k$ by setting $A_k^+ = \bigcup_n A_{k,n}^+$. 

\medskip
It will also be useful to introduce the following analogous subgroups, namely 
\[ O_{k, n+1}  = W \wr \Sym(k d^n) \hspace{.5cm}  \text{and} \hspace{.5cm} O^+_{k, n+1}  = W \wr \Alt(k d^n)\]
for all $k >0$ and $n \geq 0$, and
\[ O_{k}  = \bigcup_n O_{k, n+1} \hspace{.5cm}  \text{and} \hspace{.5cm} O^+_{k}  = \bigcup_n O^+_{k, n+1} \leq O_k.\]
Thus $O_k$ and   $O_k^+$ are open subgroups of $ \acomm(W^k)$, and are respectively contained in $A_k$ and $A_k^+$.

\begin{remark}\label{rem:A_k^+}
It is important to notice that $A_{k, n}^+ \leq A_{k, n+1}^+$ if and only if $D \leq \Alt(d)$. This implies that $[A_k : A_k^+] = 2$ if $D \leq \Alt(d)$, and $A_k^+ = A_k$ otherwise. Similarly, we have  $[O_k : O_k^+] = 2$ if $D \leq \Alt(d)$, and $O_k^+ = O_k$ otherwise.
\end{remark}

\begin{lemma}\label{lem:A_k}
We have the following.
\begin{enumerate}[\rm (i)]
\item The groups $A_k$ and $A_k^+$ (resp. $O_k$ and $O_k^+$) are non-compact and locally elliptic.

\item $A_k^+$   (resp. $O_k^+$) is topologically simple. 

%\item Any closed subgroup of finite covolume in $A_k$ contains $A_k^+$. In particular $A_k$ and $A_k^+$ do not contain any lattice subgroup.
\end{enumerate}
\end{lemma}

\begin{proof}
We shall prove the statements for $A_k$ and $A_k^+$; the arguments for $O_k$ and $O_k^+$ are similar. 
\begin{compactenum}[(i)]
\item By definition $A_k  \leq  \acomm(W^k)$ is a union of an infinite ascending chain of compact open subgroups of  $\acomm(W^k)$. The desired assertion is thus clear. 

\item 
Let $N \leq A_k^+$ be a non-trivial closed normal subgroup. Since $W^k$ has trivial quasi-centre, so does $\acomm(W^k)$ by \cite[Prop.~3.2(c)]{BarneaErshovWeigel}. Since the quasi-centre of a group contains the quasi-centre of any open subgroup, we infer that $A_k^+$ has also a trivial quasi-centre. Thus  $N$ is non-discrete. It therefore meets $\comm(W^k)$ non-trivially. In particular,  we deduce that, for $n$ large enough, the intersection $N \cap A_{k, n+1}^+$ maps onto a non-trivial normal subgroup of $\Alt(kd^n)$. Therefore, it maps onto $\Alt(kd^n)$. It follows that the same holds for $N \cap A_{k, m+1}^+$ for all $m \geq n$. This implies in particular that $N \cap A_{k, n+1}^+$ is dense in $A_{k, n+1}^+$. Since $N$ is closed, we deduce that $\bigcup_n A_{k,n}^+$ is contained in $N$, and the desired simplicity assertion follows. A similar argument shows that any non-trivial closed normal subgroup of $A_k$ contains $A_k^+$.\qedhere
%
%\item Let $H \leq A_k$ be a closed subgroup of finite covolume. For every open subgroup $O \leq A_k$, the Haar measure on $O$ is the restriction of the Haar measure on $A_k$. Therefore, the injection $O/ O \cap H \hookrightarrow A_k/H$ yields an inequality $\vol(O/O \cap H) \leq \vol(A_k/H) = V$. We apply this to the compact open subgroup $O = A_{k, n+1}$. Pushing down to the  finite quotient $\Sym(kd^{n})$, we deduce that the intersection $H \cap A_{k, n+1}$ maps onto a subgroup of $\Sym(kd^{n})$ of index~$\leq V$. Recall that $\Alt(kd^n)$ is the only non-trivial subgroup group of $\Sym(kd^{n})$ of index~$< kd^n$. Therefore, for $n$ large enough we deduce that the image of $A_{k, n+1} \cap H$ in $\Sym(kd^{n})$ contains $\Alt(kd^n)$. This implies in turn that the closure of $A_{k, n+1} \cap H$ contains $A^+_{k, n+1}$. Therefore the closed subgroup $H$ contains $A_k^+$, as claimed.\qedhere
\end{compactenum}
\end{proof}

The following result describes  the structure of the group of abstract commensurators of a self-replicating wreath branch group. We identify $V_{d, k}$ and $F_{d, k}$ with their images in $\acomm(W^k)$, see Lemma~\ref{lem:HT}.

\begin{theorem}\label{thm:CommBranch:Iwasawa}
Let $k > 0$. Then
$\acomm(W^k)$ is generated by its subgroups $V_{d,k}$ and $ A_{k,2} $, it is contained in the group $ \AAut(\mathcal T_{d, k})$ of almost automorphisms of the tree $\mathcal T_{d, k}$ and, moreover, we have the decomposition
\[ \acomm(W^k) = F_{d, k} \cdot A_k, \]
where $F_{d, k} \cap A_k = 1$. In particular  $F_{d, k}$ is a discrete subgroup of infinite covolume, and $A_k$ is a maximal locally elliptic subgroup of $\acomm(W^k)$. Moreover the closure of $V_{d, k}$ coincides with the open subgroup
$F_{d, k} \cdot O_k \leq \acomm(W^k) $.
\end{theorem}

We postpone the proof until the end of the section and first proceed to describe some consequences. 

\begin{cor}\label{cor:M(D, k)}
Let $d >1$, let $D \leq \Sym(d)$ be transitive  and set $W = W(D)$. Let moreover $k>0$ and let  $M = M(D, k)$ denote the intersection of all non-trivial closed normal subgroups of $\acomm(W^k)$. We have the following.

\begin{enumerate}[\rm (i)]
\item $M$ is a rigid group belonging to the class $\mS$.

\item $M$ is open in $\acomm(W^k)$ and contains $\comm(W^k)$.

\item $[\acomm(W^k) : M] \leq 2$.

\item $[\acomm(W^k) : M] = 2$ if and only if $d$ is odd and $D \leq \Alt(d)$. In that case we have  $M = F_{d, k} \cdot A_k^+$.
\end{enumerate}
\end{cor}

\begin{proof}
\begin{compactenum}[(i)]
\item %[(i) and (ii)] 
By Lemma~\ref{lem:A_k}, the group $A_k^+$ is topologically simple and contains $W^k$ as a compact open subgroup. Thus $\acomm(W^k) = \acomm(A_k^+)$ and, since $W^k$ has trivial quasi-centre, so does $A_k^+$. Therefore the assertion~(i) and the fact that $M$ is open follow Proposition~\ref{prop:AdmissibleRigid}.

\item[(ii)] Clear since $\comm(W^k) \leq A_k^+ \leq M$.

\item[(iii)]  By definition  $M$ is  normal in $\acomm(W^k)$. Set $V = V_{d, k} $ and identify it with its image in $\acomm(W^k)$, see Lemma~\ref{lem:HT}. 
%
%Consider the group $V \cap A_k$. It consists of those elements of $V$ which can be represented by rooted almost automorphisms of the form $(A, B, \varphi)$, where $A$ and $B$ are both balls centred at the root of $\mathcal T_{d, k}$. It is then not difficult to notice that $V \cap A_k$ is dense in $A_k$.
%
Since $V$ is dense in $F_{d, k} \cdot O_k \leq \acomm(W^k)$ by Theorem~\ref{thm:CommBranch:Iwasawa} and since $O_k$ is open, it follows that $V \cap O_k$ is dense in $O_k$. In particular $M \cap V$ contains an infinite normal subgroup of $V$.
Since the derived group $[V, V]$ is simple (see, for example, \cite[Theorem~4.16]{Brown}), it follows that $M$ contains $[V, V]$.
Since $[V, V]$ contains $F_{d, k}$, we deduce from Theorem~\ref{thm:CommBranch:Iwasawa} that $\acomm(W^k)/M \cong A_k/A_k \cap M$. Lemma~\ref{lem:A_k} guarantees that $M$ contains $A_k^+$ and the desired assertion thus follows since $[A_k, A_k^+] \leq 2$, see Remark~\ref{rem:A_k^+}.

\item[(iv)]
If $d$ is even, then $V$ is simple. Since it is dense in $F_{d, k} \cdot O_k \leq \acomm(W^k)$ and since $M$ is open and normal, we deduce that $M $ contains $F_{d,k} \cdot O_k$. Moreover $M$ contains $A_k^+$ by (iii), and since $A_k = O_k \cdot A_k^+$, we infer that $M= \acomm(W^k)$, as desired.

Assume now that $d$ is odd and that $D \leq \Sym(d)$ contains an odd permutation. Since $V \cap W^k$ is dense in $W^k$, the existence of this permutation guarantees  that $V \cap W^k$  contains an element $\sigma$ which is not annihilated in $V/ [V, V]$ (see the discussion on $V/[V, V]$ at the beginning of the present subsection above). Since $\acomm(W^k)/M$ embeds in the abelianization $V/[V, V]$, we have $\acomm(W^k) = \la \sigma \ra \cdot M$. By construction $\sigma \in M$, whence $M =\acomm(W^k)$, as desired. 

\medskip
Assume conversely that $d$ is odd and that $D \leq \Alt(d)$. Then $A_k$ contains an open subgroup $A_k^+$ of index two which is topologically simple, see Lemma~\ref{lem:A_k} and Remark~\ref{rem:A_k^+}. 

By Theorem~\ref{thm:CommBranch:Iwasawa}, we have $\acomm(W^k) = F_{d, k} \cdot A_k$. Moreover $F_{d, k}$ is contained in the derived group $[V, V]$. Therefore, we infer that the set $F_{d, k} \cdot A^+_k$ is a proper subset of $\acomm(W^k)$ which contains  $[V, V]$. Notice that $F_{d, k} \cdot A^+_k$ is open in $\acomm(W^k)$ since $A^+_k$ is so. Moreover its complement is also open, since it coincides with the translate $F_{d, k} \cdot A^+_k \cdot \sigma$ for any element $\sigma \in A_k \setminus A^+_k$. In particular $F_{d, k} \cdot A^+_k$ is closed. Thus we have highlighted a proper closed subset of $\acomm(W^k)$ which contains the subgroup generated by $[V, V] \cup A^+_k$. Since  $\acomm(W^k)$ is generated by $V \cup A^+_k$, we deduce that 
\[ M = \la [V, V] \cup A_k^+ \ra = F_{d, k} \cdot A^+_k \]
and, hence, that $M$ has index two in $\acomm(W^k)$, as desired.\qedhere
\end{compactenum}
\end{proof}

The following result is a straightforward adaptation of Theorem~4.16 from \cite{BarneaErshovWeigel}.

\begin{cor}\label{cor:EmbeddingSimple}
Assume that the finite group $D$ is perfect, and let $W = W(D)$. Then $W^k$ embeds as a compact open subgroup in some compactly generated  topologically simple group $G  \in \mathcal S$. 
%One can take $G$ to be the closure of  $[V_{d, k}, V_{d, k}]$; then  $G = F_{d, k} \cdot O^+_k$.
\end{cor}

\begin{proof}
Let $G \leq \acomm(W^k)$ be the closed subgroup generated by $[V_{d, k}, V_{d, k}]$. Thus $G$ is generated by the compact open subgroup $W^k$ and the finitely generated group $[V_{d, k}, V_{d, k}]$. In particular it is compactly generated. By Lemma~\ref{lem:JustInfinite}, the group $W$ is just-infinite. Therefore, so is $W \wr \Alt(k)$, which embeds as a compact open subgroup in $G$. As in the proof of \cite[Theorem~4.16]{BarneaErshovWeigel}, one deduces that any non-trivial closed normal subgroup of $G$ is open. Since $V_{d, k} \cap W^k$ is dense in $W^k$, it follows that any open subgroup of $G$ contains an infinite subgroup of $V_{d, k}$. The result now follows from the fact that $[V_{d, k}, V_{d, k}]$ is simple and of index at most two in $V_{d,k}$.
\end{proof}

Before undertaking the proof of Theorem~\ref{thm:CommBranch:Iwasawa}, we need some additional notation and terminology. Let $\tilde W = W \wr \Sym(k)$ and let $\tilde{\mathcal T} = \mathcal T_{d, k}$ denote the rooted tree whose root has valence $k$ and such that $\tilde{\mathcal T}_v $ is isomorphic to $\mathcal T = \mathcal T_d$ for each vertex $v$ distinct from the root. The group $\tilde W$ naturally acts on $\tilde{\mathcal T}$ by automorphisms and $\tilde W$ is branch. Moreover $\acomm(W^k) \cong \acomm(\tilde W)$. 

A pair of vertices of $\tilde{\mathcal T}$ is called \textbf{independent} if no geodesic ray emanating from the root contains both. More generally, a subset of $V(\tilde{\mathcal T})$  consisting of pairwise independent vertices is called independent. A \textbf{leaf set} is an independent subset of $V(\tilde{\mathcal T})$ which is maximal among independent sets. Given $n \geq 0$, we denote by $V_n(\tilde{\mathcal T})$ (resp. $V_{\geq n}(\tilde{\mathcal T})$) the set of vertices of level $n$ (resp. at least $n$).  Thus $V_n(\tilde{\mathcal T})$ is a leaf set.

%Before stating the main technical tool needed by the proof of Theorem~\ref{thm:CommBranch:Iwasawa}, let us adopt one more terminological convention. By definition, an abstract commensurator $\alpha$ of a profinite group $U$ is class of isomorphisms between pairs of arbitrarily small open subgroups of $U$. 

We shall need the following. 

\begin{lemma}\label{lem:Rover}
Let $n \geq 0$, $B \leq \tilde W$ be an open subgroup and $\alpha: \st_{\tilde W}(n) \to B$ be an isomorphism.
(In particular $\alpha$ defines some element of $ \acomm(\tilde W)$.)

\begin{enumerate}[\rm (i)]
\item There is a map $\Phi : V(\tilde{\mathcal T}) \to V(\tilde{\mathcal T})$ such that for each vertex $v \in V_{\geq n}(\tilde{\mathcal T})$, we have $\alpha(\rist_{\tilde W}(v)) = \rist_{\tilde W}(\Phi(v))$.

\item $\Phi$ maps every leaf set $L \subset V_{\geq n}(\tilde{\mathcal T})$ to a leaf set, and the restriction $\Phi : L \to \Phi(L)$ is injective (and hence bijective if $L$ is finite).
\end{enumerate}
\end{lemma}

\begin{proof}
We start by collecting a few preliminary observations.  By \cite[Theorem~1.2]{Rover}, the group $\acomm({\tilde W})$ is isomorphic to the relative commensurator of ${\tilde W}$ in the homeomorphism group of the boundary $\partial \tilde{\mathcal T}$. Since the isomorphism $\alpha$ defines a unique element of $\acomm({\tilde W})$, we can view it as an element of $\mathrm{Homeo}(\bd \tilde{\mathcal T})$. 

Given any subset $O \subset \bd \tilde{\mathcal T}$ containing at least two points, there is a unique vertex $v$ of maximal possible level such that $O \subset \bd \tilde{\mathcal T}_v$. We denote this vertex  by $\mathfrak v(O)$.  On the other hand, notice that the action of ${\tilde W}$ on $\partial \tilde{\mathcal T}$ is continuous, and every closed subgroup of ${\tilde W}$ has  closed orbits since it is compact. Therefore the restricted stabiliser $\rist_{\tilde W}(v)$ of a vertex $v$ has a unique orbit that consists of more than one point.
We denote this orbit by $\mathfrak o(v)$. We have $\mathfrak v( \mathfrak o(v)) = v$. 

\begin{compactenum}[(i)]
\item
Assume that $v$ has level at least $n$. This allows one to consider
\[ \Phi(v) = \mathfrak v(\alpha(\mathfrak o(v))) \,; \]
recall that $\alpha$ is regarded as a homeomorphism of $\partial\tilde{\mathcal T}$.
It is clear from the definition that $\alpha \rist_{\tilde W}(v) \alpha\inv \leq \rist_{\tilde W}(\Phi(v))$. Notice that $\mathfrak o(\Phi(v)) = \alpha(\mathfrak o(v))$. Therefore $\mathfrak o(v)$ is the unique orbit of $\alpha\inv \rist_{\tilde W}(\Phi(v)) \alpha$ that consists of more than one point.
We infer that $\alpha\inv \rist_{\tilde W}(\Phi(v)) \alpha \leq \rist_{\tilde W}(v)$, whence $\alpha \rist_{\tilde W}(v) \alpha\inv = \rist_{\tilde W}(\Phi(v))$ as desired.

\item
A pair of vertices $\{v, v'\} \subset  V(\tilde{\mathcal T})$ is independent if and only if $\mathfrak o(v)$ and $\mathfrak o(v')$ are disjoint. Moreover a finite independent subset $L  \subset  V(\tilde{\mathcal T})$ is a leaf set if and only if $\bd \tilde{\mathcal T} = \bigcup_{v \in L} \mathfrak o(v)$. It readily follows that $\Phi$ maps independent sets to independent sets, and leaf sets to leaf sets. Finally, given a leaf set $L \subset  V_{\geq n}(\tilde{\mathcal T})$, then $\{\Phi(v), \Phi(v')\}$ is an independent pair for all $v \neq v' \in L$. In particular the restriction of $\Phi$ to $L$ is injective. 
\qedhere
\end{compactenum}
\end{proof}

Notice that the lemma implies in particular that every germ of automorphism of $W^k$ acts by almost automophism on $\tilde{\mathcal T} = \mathcal T_{d, k}$.

\begin{proof}[Proof of Theorem~\ref{thm:CommBranch:Iwasawa}]
Every abstract commensurator of $\tilde W$ admits a representative of the form $\alpha: \st_{\tilde W}(n) \to B$ for some $n \geq 0$ and some open subgroup $B \leq \tilde W$. We now invoke Lemma~\ref{lem:Rover} to the leaf set $L = V_n(\tilde{\mathcal T})$. In view of the definition of the embeddings $F_{d, k} \subset V_{d,k} \subset \acomm(W^k)$ (see Lemma~\ref{lem:HT}), we deduce that there exists an element $g \in F_{d, k}$ such that $\alpha g$ normalises  $ \st_{\tilde W}(n)$. Since $\Aut(\st_{\tilde W}(n))$ is contained in $A_k$ (see Lemma~\ref{lem:Aut(W)}) we deduce the equality 
\[ \acomm(W^k) = \acomm(\tilde W) =  A_k\cdot F_{d, k}  =   F_{d, k} \cdot A_k, \]
as well as the fact that the group $\acomm(W^k)$ consists entirely of almost automorphisms of the tree $\mathcal T_{d, k}$. 

Given $v \in V_{\geq 2}(\tilde{\mathcal T})$, there exist $\alpha \in V_{d, k} $ such that $\Phi(v) $ has level exactly two, where $\Phi$ is defined by Lemma~\ref{lem:Rover}. Moreover, for all $n$ and each vertex $v$ of level~$n$, we have seen that $A_{k,n} = \Aut(\rist_{\tilde W}(v)) \wr  \Sym(k d^{n-1})$. The equality $\acomm(\tilde W) =  \la V_{d, k}, A_{k,2} \ra$ now follows from the fact that $\Sym(k d^{n-1})$ is in fact contained in $V_{d, k}$.

That $F_{d,k} \cap A_k = 1$ follows from the definition of $F_{d,k}$.  

We next consider the group $V \cap A_k$. It consists of those elements of $V$ which can be represented by rooted almost automorphisms of the form $(A, B, \varphi)$, where $A$ and $B$ are both balls centred at the root of $\mathcal T_{d, k}$. It is then not difficult to notice that $V \cap O_k$ is dense in $O_k$. Since $F_{d, k}$ is contained in $V_{d, k}$, we deduce from the equality $\acomm(W^k) = F_{d, k} \cdot A_k$ that $V_{d, k}$ is dense in $F_{d,k} \cdot O_k  \leq \acomm(W^k)$, as desired.

Finally, combining the triviality of $F_{d,k} \cap A_k$ with the fact that $A_k$ is open, we deduce that the covolume of $F_{d,k}$ in $\acomm(W^k)$ is bounded below by the volume of $A_k$, which is infinite since $A_k$ is not compact (see Lemma~\ref{lem:A_k}). Since $F_{d, k}$ is discrete and torsion-free, no subgroup of $\acomm(W^k)$ containing $A_k$ properly can be locally elliptic. 
\end{proof}

\subsection{Embeddings in compactly generated rigid simple groups}

We have seen in Corollary~\ref{cor:EmbeddingSimple} that if $D$ is perfect, then the wreath branch group $W(D)$ embeds as a compact open subgroup in some compactly generated simple group. Similarly, if $D$ coincides with a point stabiliser in some finite transitive permutation  group $F  \leq   \Sym(d+1)$, then $W(D)^2$ embeds as an edge-stabiliser in the Burger--Mozes simple group $U(F)^+$.

Our next goal is to determine when $W(D)^k$ embeds in a compactly generated simple group  which is \emph{rigid}. Recall that a rigid simple group $M = M(D, k)$ containing $W(D)^k$ as a compact open subgroup was defined in Corollary~\ref{cor:M(D, k)}. By Proposition~\ref{prop:AdmissibleRigid}(i), it is, up to isomorphism, the unique topologically simple locally compact group satisfying this condition.

\begin{theorem}\label{thm:CommBranch}
Let $D  \leq  \Sym(d)$ be transitive, let $W=W(D)$ and let $M  \in \mS$ be a rigid simple group containing a compact open subgroup isomorphic to $W^k$ for some $k>0$. Then the following conditions are equivalent.
\begin{enumerate}[\rm (i)]
\item $M$ is compactly generated.

\item $\mathscr L(M)$ is compactly generated.

\item $\mathscr L(M)$ is compactly generated and $[\mathscr L(M): M]\leq 2$.

\item $\mathscr L(M) = V_{d, k} \cdot W^k$.

\item $\Out(W) = 1$.

\item $\norma_{\Sym(d)}(D) =D$.
\end{enumerate}
\end{theorem}

\begin{proof}%[Proof of Theorem~\ref{thm:CommBranch}]
\begin{asparaitem}\itemsep1ex
    \item[(v) $\Leftrightarrow$ (vi)]
	Immediate from Lemma~\ref{lem:Aut(W)}.

    \item[(iii) $\Rightarrow$ (i) and (ii)]
	Obvious.

    \item[(i) $\Rightarrow$ (v) ]
	By Lemma~\ref{lem:Aut(W)}, the group $W$ is saturated and, hence,  the subgroup $\st_W(1)$ is characteristic in $W$. Moreover,
	$\big(\st_W(1)\big)^k$ is characteristic in $W^k$.
	Also, $\Out(W)$ is either trivial or uncountable. 

	It is straightforward to deduce from the description of the groups $\Aut(W)$ and $\Aut(W^k)$ given in Lemma~\ref{lem:Aut(W)} that if $\Out(W)$ is non-trivial, then the image of $\Aut(W^k)$ in $\Aut\big( (\st_W(1))^k\big) $ is of uncountable index. In view of Proposition~\ref{prop:Aut},   the hypothesis that $M$ is rigid and compactly generated guarantees that this index is at most countable. Therefore $\Out(W)$ must be trivial, as desired.

    \item[(ii) $\Rightarrow$ (v)]
	Assume that $\Out(W)$ is not trivial. Then it is uncountable, and so is $\Out(W^k)$, see Lemma~\ref{lem:Aut(W)}. This implies that the index of $W^k$ in $\mathscr L(M)$ is uncountable. In particular $\mathscr L(M)$ cannot be $\sigma$-compact, and \emph{a fortiori} not compactly generated.

    \item[(v) $\Rightarrow$ (iv)]
Since $A_{k, n} \cap V_{d, k}$ maps onto $\Sym(kd^n)$ for all $n>0$,
	we deduce from Theorem~\ref{thm:CommBranch:Iwasawa} that $\mathscr L(M) = \Gamma \cdot W^k$, where $\Gamma = V_{d, k}$.
	The claim follows.

    \item[(iv) $\Rightarrow$ (iii)]
	Clearly the assumption implies that $\mathscr L(M)$ is compactly generated.
The inequality $[\mathscr L(M):M] \leq 2$ follows from Corollary~\ref{cor:M(D, k)}.\qedhere
\end{asparaitem}

\end{proof}

%The following supplement to Theorem~\ref{thm:CommBranch} clarifies when exactly the group $\mathscr L(M)$ has non-trivial abelianisation. 
%???

We conclude this section with the proof of the last two theorems announced in the introduction.

\begin{proof}[Proof of Theorem~\ref{thmi:L(G)}]
We deduce from Proposition~\ref{prop:AdmissibleRigid} that $G$ embeds as an open subgroup in some rigid simple group $\tilde G \in \mS$. Let $D$ denote the point-stabiliser $F_0$ viewed as a subgroup of $\Sym(d)$. Set $W =W(D)$ and recall that every edge-stabiliser in $G$ is isomorphic to $W^2 = W \times W$. 

It follows from  these observations that the hypotheses of Theorem~\ref{thm:CommBranch} are satisfied with $k=2$. All the desired statements now follow, except the fact that $H = [\mathscr L(G), \mathscr L(G)] $ is abstractly simple. The passage from topological simplicity to abstract simplicity goes as follows. 

First recall that $\comm(G)$ is open in $\acomm(G)$. Moreover $G$ is abstractly simple (see Theorem~\ref{th:tits}) and is thus contained in $H$. This implies that $H$ is itself open. Let now $N$ be an abstract normal subgroup of $H$. The intersection $G \cap N$ is either trivial or equal to $G$. Since $G$ is open, the former case implies that $N$ is discrete while the latter case implies that $N$ is open.  In either case $N$ is closed. The result follows.
\end{proof}

We finally prove Theorem~\ref{thmi:NewSimpleGroups}, or rather the following more detailed statement.
\begin{theorem}\label{thm:NewSimpleGroups}
Let $d > 1$,  $D \leq \Sym(d)$ be transitive and $W = W(D)$ be the profinite branch group defined as the infinitely iterated wreath product of $D$ with itself.  Then for every $k>0$, there is a locally compact group $M = M(D, k)$ which is topologically simple and rigid, and which contains the direct product $W^k$ of $k$ copies of $W$ as a compact open subgroup. Moreover:
\begin{enumerate}[\rm (i)]
\item $M$ is uniquely determined up to isomorphism. 

\item $M$ is compactly generated if and only if  $\norma_{\Sym(d)}(D) = D$. 

\item $[\acomm(M): \comm(M)] \leq 2$ and $[\acomm(M): \comm(M)] = 2$ if and only if $d$ is odd and $D  \leq \Alt(d)$. 

\item $\acomm(M) =  F_{d, k} \cdot A_k$, where  $F_{d, k}$ is a copy of the Higman--Thompson group  embedded as a discrete subgroup, and $A_k$ is a non-compact locally elliptic open subgroup  such that $F_{d, k} \cap A_k = 1$. Moreover $A_k$ is a maximal locally elliptic subgroup of $\acomm(M)$.  Furthermore $F_{d, k}$ is contained in $\comm(M)$ and $A_k$ possesses an open subgroup $A_k^+$ of index at most two which is topologically simple and equally contained in $\comm(M)$. 

\item If  $k  \equiv k' \mod (d-1)$, then  $M(D, k) \cong M(D, k')$.

\item If $D' \leq \Sym(d')$ contains a simple subgroup which does not embed as a subgroup of $D$, then  $M(D, k) \not \cong M(D', k')$ for all $k, k' >0$. 

\end{enumerate}
\end{theorem}

\begin{proof}
The existence of a rigid $M = M(D,k) \in \mS$ follows from Corollary~\ref{cor:M(D, k)}. Assertion~(i) follows from Proposition~\ref{prop:AdmissibleRigid}, Assertion~(ii) is contained in Theorem~\ref{thm:CommBranch} while Assertions~(iii) and~(iv) follow from Theorem~\ref{thm:CommBranch:Iwasawa} and~Corollary~\ref{cor:M(D, k)}.

\medskip
In order to prove Assertions~(v) and~(vi), first notice that groups of the form $M(D, k)$ are all rigid. In particular two of them are isomorphic if and only if they are locally isomorphic. Assertion~(v) is now straightforward to establish. For Assertion~(vi), it suffices to show that if $D'$ contains a simple subgroup $S$ which does not embed as a subgroup of $D$, then $W(D)^k$ and $W(D')^{k'}$ are not locally isomorphic.  It is not difficult to show that every identity neighbourhood of $W(D')^{k'}$ contains a finite subgroup isomorphic $D'$. On the other hand, one shows that no finite quotient of $W(D)^k$ contains a subgroup isomorphic to $S$, which implies that the profinite group $W(D)^k$ itself does not contain any copy of $S$. Thus $W(D)^k$ and $W(D')^{k'}$ are not locally isomorphic, as claimed.
\end{proof}

\begin{remark}
None of the compactly generated simple groups appearing in Theorem~\ref{thm:NewSimpleGroups} (or in Corollary~\ref{cor:EmbeddingSimple}) admits any continuous, proper and \emph{cocompact} action on any locally finite tree, or on any locally finite CAT(0) cell complex. Indeed, the Higman--Thompson group $V_{d, k}$, as well as its derived group $[V_{d, k}, V_{d, k}]$, contains a copy of every finite group. This implies that if $V_{d, k}$ acts on a CAT(0) cell complex $X$, then the size of links of vertices in $X$ is unbounded, thereby preventing the action of any group containing $V_{d, k}$ from being cocompact.  
\end{remark}

%=========================================================================================

\end{document}